\documentclass[11pt]{amsart}
\usepackage[utf8]{inputenc}
\usepackage[left=2cm, right=2cm, top=2cm, bottom=2cm]{geometry}
\usepackage{mdframed, mathtools, esint, amsmath, enumitem, amscd, amssymb, indentfirst, latexsym, multicol, verbatim, mdframed, amsthm, esint}
\usepackage[ddmmyyyy]{datetime}
\usepackage[dvipsnames,table,xcdraw]{xcolor}
\usepackage[T1]{fontenc}
\usepackage[sort, numbers]{natbib}
\usepackage{ulem}
\usepackage{tikz}
\usetikzlibrary{arrows.meta, positioning}

\newcommand{\R}{\mathbb{R}}

\newcommand*\dd{\mathop{}\!\mathrm{d}}

\renewcommand{\mapsto}{\longmapsto}

\renewcommand{\epsilon}{\varepsilon}

\newlist{steps}{enumerate}{1}
\setlist[steps, 1]{label = Step \arabic*:}

\newtheorem{theo}{Theorem}[section]
\newtheorem{prop}[theo]{Proposition}
\newtheorem{lem}[theo]{Lemma}
\newtheorem{rem}[theo]{Remark}

\newtheorem{defi}[theo]{Definition}
\usepackage[utf8]{inputenc}

\newcommand{\sgn}{\mathrm{sgn}}

\def \n {\nabla}

\def \O {\Omega}

\def \cD {\mathcal{D}}
\def \cE {\mathcal{E}}

\def \p {\partial}

\def \dd  {\, \mathrm{d}}

\def \dm  {\, \mathrm{d}m}

\def \dt  {\, \mathrm{d}t}

\def \dx  {\, \mathrm{d}x}
\def \dy  {\, \mathrm{d}y}
\def \dz  {\, \mathrm{d}z}

\DeclareMathOperator{\diam}{diam} %
\DeclareMathOperator{\supp}{supp} %
\DeclareMathOperator{\loc}{loc} %

\title{Topological and Purely Topological Alignment Dynamics}
\author{Trevor M. Leslie}
\address{Department of Applied Mathematics,
	Illinois Institute of Technology}
\email{tleslie@illinoistech.edu}
\author{Jan Peszek}
\address{Institute of Applied Mathematics and Mechanics, University of Warsaw}
\email{j.peszek@uw.edu.pl}
\date{}

\begin{document}

	\maketitle
	\begin{abstract}
		We study the Euler Alignment system of collective behavior, equipped with `topological' interaction protocols, which were introduced to the mathematical literature by Shvydkoy and Tadmor.  Interactions subject to these protocols may depend on both the Euclidean distance between agents and on the mass distribution between them---the `topological' component.  When the interaction protocol is regular, we prove sufficient conditions for the existence of global-in-time classical solutions, related to the initial nonnegativity of a conserved quantity of the system.  The remainder of our results explore the case where the interactions are `purely' topological and the interactions do not depend on the Euclidean distance.  We show that in this case, the system decouples into an autonomous velocity equation in mass coordinates together with a scalar conservation law with time-dependent flux determined by the velocity.  We analyze the long-time behavior for the dynamics associated to both regular and singular protocols.  
	\end{abstract}
	
	\section{Introduction}
	
	\subsection{The Euler Alignment System}
	
	This paper is concerned with the Euler Alignment system, which when posed on $\mathbb{R}^d$ takes the following general form.
	\begin{equation}
		\label{e:EA}
		\left\{ 
		\begin{split}
			\partial_t\rho + \nabla\cdot(\rho u) & = 0,
			\qquad (x,t)\in \mathbb{R}^d\times (0,\infty), \\
			\partial_t (\rho u) + \n\cdot(\rho u \otimes u) & = \int_{\mathbb{R}^d} \phi(x,y,t)(u(y,t) - u(x,t)) \rho(x,t) \rho(y,t) \dy 
		\end{split}
		\right.	
	\end{equation}
	Here $\rho=\rho(x,t)\ge 0$ and $u=u(x,t)\in \mathbb{R}^d$ denote the density and macroscopic velocity of a continuum of agents which interact according to the `communication protocol' $\phi\ge 0$.  We make the global assumption that $\rho$ is normalized to have mass $1$:
	\[
	\int_{\mathbb{R}^d} \rho(x,\cdot)\dx = 1.
	\]
	The character of the system \eqref{e:EA} and evolution of the associated dynamics depend strongly on the properties of $\phi$.  Most of the literature on \eqref{e:EA} treats the case where $\phi = \phi(|x-y|)$ is a function of the Euclidean distance between $x$ and $y$ only (possibly singular at zero, in which case the integral in \eqref{e:EA} may be understood in a principal value sense if necessary). 
	But $\phi$ may have other dependencies as well.  In particular, this paper is concerned with the case where $\phi(x,y,t)$ may depend on the mass distribution between $x$ and~$y$, a paradigm that was first advocated for in the mathematical literature by Shvydkoy and Tadmor  \cite{STtopo}, who referred to such $\phi$ as `topological protocols.' In fact, the majority of our work is focused on a `purely topological' variant of the protocols introduced in \cite{STtopo}, where $\phi$ depends \textit{only} on the mass distribution between points and not on their Euclidean distance.  Interestingly, the use of purely topological protocols leads to a complete decoupling of the velocity equation in mass coordinates, which enables the study of the long-time behavior in terms of an autonomous velocity equation.
	
	The Euler Alignment system was first derived by Ha and Tadmor in \cite{HT2008} as a hydrodynamic version of the celebrated Cucker--Smale system of ODEs \cite{CS2007a, CS2007b}, which reads as follows. For $N$ agents with masses, positions, and velocities given by $(m_i, x_i, v_i)_{i=1}^N$, 
	\begin{equation}
		\label{e:CS}
		\left\{ 
		\begin{split} 
			\dot{x}_i & = v_i, \qquad i = 1, \ldots, N,\\
			\dot{v}_i & = \sum_{\substack{j=1 \\ j\ne i}}^N m_j \phi(x_i, x_j, t)(v_j - v_i).
		\end{split} 	
		\right. 
	\end{equation}
	System \eqref{e:EA} is the pressureless (monokinetic) variant of macroscopic alignment dynamics derived from the Cucker--Smale model cf. \cite{FigalliKang2019, FP23} (alternative non-monokinetic approaches can be found e.g. in \cite{KMT2015}).
	
	The study of \eqref{e:EA} and \eqref{e:CS} falls under the umbrella of \textit{collective dynamics}, which is concerned with understanding how large-scale macroscopic structures emerge from local interactions in systems of many agents.  Formulating and analyzing mathematical models that capture such phenomena is in general highly nontrivial.  The Cucker--Smale and Euler Alignment systems have received copious attention because of their compatibility with `flocking' dynamics (terminology intentionally reminiscent of a flock of birds) and their relative mathematical tractability; analysis of these systems is nonetheless quite subtle, and a good deal remains unknown.  We will define precisely what we mean by `flocking' below and discuss a select few references, but it is not our purpose here to review the already vast literature on Cucker--Smale-type systems.  We refer instead to the survey papers \cite{MT2014, ShvydkoySurvey, T-21}, the book \cite{ShvydkoyBook}, and references therein, for more on these systems and some of their applications. 
	
	\subsection{Flocking Dynamics and the Role of Topological Distance}
	The communication protocols $\phi(x,y,t)$ we consider will be symmetric in $x$ and $y$, which will guarantee conservation of the total momentum of the system:
	\[
	\int_{\mathbb{R}^d} \rho u(t)\dx = \overline{u} :=\int_{\mathbb{R}^d} \rho_0 u_0 \dx,
	\qquad t\ge 0.
	\]
	For our version of the system, then, it makes sense to say that the system \eqref{e:EA} undergoes `velocity alignment' in the time-asymptotic limit if $u(t)$ tends in some sense to its density-weighted average $\overline{u}$ as $t\to \infty$.  (In the absence of guaranteed momentum conservation, a suitable replacement would be to require that $\diam u(t)\to 0$.)  We say that `flocking' occurs if the density profile converges (in some appropriate topology) to a traveling wave in the time-asymptotic limit: $\rho(\cdot +\overline{u}t,t) \stackrel{t\to +\infty}{\longrightarrow} \rho_\infty$.  There is a robust literature detailing and proving various conditions which are sufficient to guarantee that velocity alignment and flocking occur.  Many of them assume that the communication protocol is radial, $\phi = \phi(|x-y|)$, and  `heavy-tailed', i.e., nonintegrable at infinity, and establish results of the form `smooth solutions must flock' under this heavy-tail assumption, cf. \cite{HaLiu2008} and many later references.  It is more realistic, however, to assume that distant agents do not communicate directly with one another---i.e., that the protocol is `local'---and thus that $\phi(x,y,t)$ should be zero if $x$ and $y$ are sufficiently far apart.  For such local protocols, the only mechanism through which information can propagate longer distances is through \textit{in}direct communication, which requires a `chain' of communicating agents, or, in the context of \eqref{e:EA}, some sort of control from below on the density or a related quantity.  
	
	Establishing alignment through chain-connectivity arguments is a delicate matter; we mention two papers where such arguments appear.  One is the work \cite{MPT2019} by Morales, the second author, and Tadmor, where it is assumed that the communication protocol is extremely strong (relative to features of the initial configuration $(\rho_0, u_0)$) for small $|x-y|$.  The other work, which is more relevant to our present considerations, is due to Shvydkoy and Tadmor \cite{STtopo}, and it takes a complementary approach.  	Rather than assuming once and for all that the communication is sufficiently strong, the authors build into their protocol a mechanism that `boosts' $\phi$ in regions where $\rho$ is small---precisely where it is needed most.  More concretely, \cite{STtopo} introduces a dependence of $\phi(x,y,t)$ on the so-called \textit{topological distance} $\dd_{\rho(t)}(x,y)$ between $x$ and $y$ at time $t$, which they define to be the integral of $\rho(t)$ over the `communication domain' $\Omega(x,y)$ that lies between them.  There is not a canonical choice for the exact shape of $\Omega(x,y)$ (cf. \cite{ReynoldsShvydkoy2020}), but one natural approach is to take $\Omega(x,y)$ to be an American football-shaped region whose two ends are located at $x$ and $y$. The relevance of the topological distance is motivated by field studies, e.g., the StarFlag project (cf. \cite{BalleriniTopo2008}), and the principle that it encodes is that information travels more slowly in regions of higher density.  In other words, intermediate agents cause interference in communication, and thus agents with more mass between them should be considered `farther apart'. On the other hand, the information flows undisturbed through regions with low density, similar to the case of Cucker--Dong model with communication reaching $q$-closest neighbors \cite{CD-16, fuzzy, MinMuchPe2020}.  
	
	With respect to the principles outlined above, \cite{STtopo} offers a compelling narrative.  It provides a criterion on $\rho(t)$ that is sufficient to guarantee that alignment occurs, and then demonstrates how that criterion relaxes when topological distance is incorporated into the communication protocol.  Thus not only is the topological distance physically relevant, but furthermore, its inclusion in the model actively serves to promote flocking behavior.  Our analysis below will showcase some additional features of the topological dynamics, including the extreme case where the communication protocol depends \textit{only} on the topological distance between points.  	
	
	\subsection{Well-posedness and the Role of Dimension 1}
	
	It is well known that systems with the left-hand side like that of \eqref{e:EA} are dramatically harder to analyze in higher dimensions than they are when the spatial dimension is 1 (at least, from the point of view of well-posedness).  The Euler Alignment system is no exception.  Suppose for a moment that $\phi = \phi(|x-y|)$ is radial.  Then \eqref{e:EA} admits an additional conserved quantity (call it $e$) satisfying, in 1D, a pure continuity equation $\partial_t e + \partial_x(ue) = 0$, cf. \cite{TT2014, CCTT2016,ShvydkoyTadmorI, DKRT}, etc.  This quantity $e$ is the cornerstone of essentially all existing well-posedness theory for \eqref{e:EA} in one spatial dimension.  However, no natural analog of $e$ is known to satisfy a similarly simple equation in higher dimensions, which is part of why the state of the art in 1D is much more developed than the corresponding theory in higher dimensions for this specific system. 
	
	In 1D, the communication domain $\Omega(x,y)$ in Shvydkoy and Tadmor's definition of the topological distance degenerates to an interval $(x,y)$, so that for $\rho\in L^1(\mathbb{R})$ we have
	\[
	\dd_{\rho}(x,y) = \bigg| \int_x^y \rho(z)\dz\bigg|
	= |M(y) - M(x)|,     \qquad M(z) = \int_{-\infty}^z \rho(z')\dz',
	\]
	and a `topological kernel' will take the form
	\begin{equation}
		\label{e:topokernel}
		\phi(x,y,t) = \varphi(\!\dd_{\rho(t)}(x,y), x-y)
	\end{equation}
	for some $\varphi:[0,1]\times \mathbb{R}\to [0,\infty)$.  In \cite{STtopo}, it is shown that, perhaps surprisingly, the Euler Alignment system equipped with $\varphi(\!\dd,z) = \frac{h(|z|)}{\dd^\tau |z|^{1+\alpha - \tau}}$ (with $h$ smooth, $0<\alpha<2$, and $\tau>0$) still admits an `$e$-quantity'; this quantity is crucial for their theory of existence and uniqueness of 1D strong solutions in the periodic, vacuum-free setting.  
	
	In Section \ref{s:threshold} below, we demonstrate that the Euler Alignment system corresponding to any $\varphi(\!\dd, z)$ satisfying some mild regularity properties (including boundedness) also admits an $e$-quantity.  We will establish that global-in-time classical solutions exist when this quantity is nonnegative at time zero; this is essentially an extension of the positive direction of the so-called `critical threshold condition' of Carrillo, Choi, Tadmor, and Tan~\cite{CCTT2016}, which applies to bounded, Euclidean kernels $\phi = \phi(|x-y|)$.  Unlike the case considered in \cite{CCTT2016}, it is not clear whether $e_0\ge 0$ is a \textit{sharp} condition in our case; the topological contribution seems to promote well-posedness, but it is difficult to quantify exactly by how much.  
	
	\subsection{Purely Topological Protocols and Reformulation of the System}
	
	After establishing global-in-time existence of classical solutions under the threshold condition $e_0\ge 0$, we will concern ourselves with an additional simplification in 1D of the system, which arises in the very special but instructive case where $\phi$ is `purely topological', i.e., $\varphi(\!\dd, z) = \varphi(\!\dd)$.  More precisely, we assume that $\varphi:[0,1]\to \mathbb{R}\cup\{+\infty\}$ is nonnegative, possibly singular at zero (but nowhere else), and that it is at least locally Lipschitz continuous away from $0$.  The communication protocol $\phi(x,y,t)$ then takes the form
	\begin{equation}
		\label{e:puretopophi}
		\phi(x,y,t) = \varphi(\!\dd_{\rho(t)}(x,y)).	
	\end{equation}
	For reference, we explicitly record the system \eqref{e:EA} specialized to the case of dimension 1 and $\phi$ as in \eqref{e:puretopophi}:
	
	\begin{equation}
		\label{e:EApuretopo}
		\left\{ 
		\begin{split}
			\partial_t\rho + \partial_x(\rho u) & = 0,
			\qquad (x,t)\in \mathbb{R}\times (0,\infty), \\
			\partial_t (\rho u) + \partial_x(\rho u^2) & = -\int_{\mathbb{R}} \varphi(\!\dd_{\rho(t)}(x,y))(u(x,t) - u(y,t)) \rho(x,t) \rho(y,t) \dy .
		\end{split}
		\right.	
	\end{equation}
	Let us reformulate it formally in mass coordinates (see e.g. \cite{MUCHAMASA1, MUCHAMASA2}). Suppose $(\rho, u)$ is a non-vacuous and sufficiently smooth solution of \eqref{e:EApuretopo}; let $M(\cdot,t)$ be a cumulative distribution of $\rho(\cdot, t)$, and define $v:(0,1)\to \mathbb{R}$ via $v(M(x,t),t) = u(x,t)$.  Then we have
	\begin{align*}
		\p_t v(M(x,t), t)
		& = \frac{\dd}{\dt} v(M(x,t),t) - \p_t M(x,t) \p_m v(M(x,t),t) \\
		& = \p_t u(x,t) + u(x,t)\p_x M(x,t) \p_m v(M(x,t),t) 
		= (\p_t u + u \p_x u)(x,t) \\
		& = -\int_\R \varphi(\!\dd_{\rho(t)}(x,y))(u(x,t) - u(y,t))\rho(y,t)\dy \\
		& = -\int_\R \varphi(|M(x,t) - M(y,t)|)(v(M(x,t),t) - v(M(y,t),t)) \p_y M(y,t)\dy \\
		& = -\int_0^1 \varphi(|M(x,t)-m'|)(v(M(x,t),t) - v(m',t))\dm',
	\end{align*}
	which leads to
	\begin{equation}
		\label{e:reformulation}
		\begin{cases}
			\p_t v(m,t) = -\displaystyle\int_0^1 \varphi(|m-m'|)(v(m,t)-v(m',t))\dm',\qquad (m,t)\in (0,1)\times (0,\infty), \\
			\p_t M(x,t) + \p_x(A(M(x,t),t)) = 0,
			\qquad A(m,t) = \int_0^m v(m',t)\dm'
		\end{cases}
	\end{equation}
	The point of this reformulation is that the unknown $v$ decouples completely. Once a solution for \eqref{e:reformulation}$_1$ is found, the flux $A$ can be computed and \eqref{e:reformulation}$_2$ solved, using the theory of scalar conservation laws. 
	We perform a rigorous analysis of equivalence between \eqref{e:EApuretopo} and \eqref{e:reformulation} in Section \ref{s:Equivalence}.
	
	\subsection{Outline of the Paper and Statement of Results}
	\label{ss:SOR}
	
	This paper initiates the study of four main questions, the last three of which concern the case of purely topological kernels.
	\begin{enumerate}[label = (\roman*)]
		\item Which initial data $(\rho_0, u_0)$ lead to global-in-time classical solutions for the Euler Alignment system with bounded topological protocols?
		\item In what sense is the system \eqref{e:reformulation} well-posed, and how does the well-posedness depend on $\varphi$?
		\item When exactly does a solution $(v,M)$ of \eqref{e:reformulation} generate a pair $(\rho, u)$ that satisfies \eqref{e:EApuretopo} in a meaningful sense?
		\item Finally, how do solutions $(v,M)$ of \eqref{e:reformulation} behave, and how does this behavior transfer to the corresponding pairs $(\rho, u)$?  In particular, for which kernels $\varphi$ are stability,  velocity alignment and flocking guaranteed for any initial data?
	\end{enumerate}
	
	These questions are of course intertwined, but we basically address question (i) in Section \ref{s:threshold}, question (iii) in Section \ref{s:Equivalence}, question (ii) in Section \ref{s:vsolutions}, and question (iv) in Section \ref{s:LongTime}.  We give a brief summary.  
	
	\medskip
	
	\subsubsection{A Threshold Condition for Global Existence of Classical Solutions}  As mentioned above, our first result concerns classical solutions of the system with regular topological protocols.
	\begin{theo}
		\label{t:classical}
		Consider the Euler Alignment system \eqref{e:EA} in one spatial dimension, with communication protocol of the form \eqref{e:topokernel}, where $\varphi = \varphi(\!\dd,z)$ is Lipschitz and nonincreasing in $\dd$, as well as even and bounded in~$z$.  
		Assume $\rho_0, u_0\in C^1(\mathbb{R})$ and define $\psi_0:\mathbb{R}\to \mathbb{R}$ via
		\begin{equation}
			\label{e:psi0def}
			\psi_0(\alpha) := u_0(\alpha) + \int_\R \Phi \big(\!\dd_{\rho_0}(\alpha, \gamma), \alpha -\gamma\big)\rho_0(\gamma) \dd\gamma, \qquad 
			\Phi(\!\dd,z) = \int_0^z \varphi(\!\dd,\zeta)\dd \zeta
		\end{equation}    
		If $\psi_0$ is nondecreasing, then there exists a unique global-in-time classical solution associated to the initial data $(\rho_0, u_0)$. 
	\end{theo}
	The proof of Theorem \ref{t:classical} is our first goal and can be found in Section \ref{s:threshold}.
	
	\begin{rem}\rm
		Interestingly, the threshold for global-in-time classical solutions of purely topological dynamics can be formulated in mass coordinates in a way that is completely independent of the initial density~$\rho_0$; see Remark~\ref{rem:ddd}. This feature sharply distinguishes the purely topological setting from most known critical threshold results for Euler–alignment and porous medium–type dynamics, where the density plays a fundamental role.
	\end{rem}
	
	We remark that the uniqueness part of the assertion above should be understood to mean that $\rho$ is uniquely determined, and that $u(t)$ is uniquely determined in $\supp \rho(t)$ for each $t\ge 0$.  This version of the Theorem is stated in slightly simplified form; in fact, we really only need $\psi_0$ to be nondecreasing inside the set $U:=\{x:\rho_0(x)>0\}$. It is also possible to include, for instance, the case where $U$ consists of finitely many open intervals and $\rho_0$ and $u_0$ are not required to be $C^1$ (or even defined) outside $U$. Our proof of Theorem \ref{t:classical} follows the approach of the first author in \cite{L2019CTC}, see also \cite{LLST2020geometric}.  The `$e$-quantity' emerges as the spatial derivative of a time-dependent version of $\psi_0$, and indeed, our criterion for global-in-time existence can be written as $e_0:=\partial_x \psi_0 \ge 0$.  We will give a precise formula for $\psi$ in Section \ref{s:threshold}, and we will also discuss a partial converse of Theorem \ref{t:classical} in Remark~\ref{r:classicalnegative}.
	
	In the microscopic and mesoscopic (kinetic) level, transformations akin to \eqref{e:psi0def} serve as a singularity reduction of the kernel $\varphi$  in kinetic descriptions of alignment dynamics. In \cite{ChoiZhang2021, PeszekPoyato2022kinetic} (see also \cite{HaKimParkZhang2019, K-21}) it is shown that the second-order Cucker--Smale-type dynamics with communication $\widetilde{\varphi}(r)=r^{-\alpha}$ with $\alpha\in(0,1)$ can be equivalently reformulated as first-order dynamics with communication $\widetilde{\Phi}(r)\approx r^{1-\alpha}$. (Here $r$ refers to Euclidean distance.) Article \cite{PP-23} reveals a gradient flow structure of such models. 
	
	\subsubsection{Types of Solutions}
	We will define and treat several types of solutions of the systems \eqref{e:EApuretopo} and \eqref{e:reformulation}, described presently.  (See also the diagram below.)  More precise definitions appear below in Section~\ref{s:Equivalence}.
	First, a `weak solution' $(v,M)$ of \eqref{e:reformulation} will simply mean a pair $(v,M)$ such that $v$ is a weak solution of~\eqref{e:reformulation}$_1$ and $M$ is an entropy solution of the scalar conservation law. 
	When we say that $(\rho, u)$ is a `mass-distributional solution' of \eqref{e:EApuretopo}, we will mean that there is a weak solution $(v,M)$ of \eqref{e:reformulation} such that $\partial_x M = \rho$ and $u = v\circ M$.  The pair $(\rho, u)$ is not a priori required to satisfy any equation in this case; rather, the term `mass-distributional solution' should be considered a relaxation of the usual notion of solution that shifts the burden of `solution' to the cumulative mass-distribution function $M$ of $\rho$ and the velocity $v$ in mass coordinates.    Finally, we say that $(v,M)$ is an \textit{admissible} solution if it is a weak solution of~\eqref{e:reformulation} and if $M$ is absolutely continuous in the spatial variable, for almost every time.  Admissible solutions $(v,M)$ are so named because they are in natural correspondence with pairs $(\rho, u)$ that satisfy the Euler Alignment system \eqref{e:EApuretopo} in a certain sense to be described below---see Definition \ref{def:admist}, where we provide a self-contained definition of `admissible solutions' $(\rho, u)$ of the system \eqref{e:EApuretopo}, as well as Proposition \ref{p:equivalence}, where we state the correspondence precisely.
	
	We can make meaningful statements about the long-time behavior of admissible solutions, as we demonstrate in Section \ref{s:LongTime}.  The other situation in which we can provide a rather straightforward analysis of the long-time behavior is the one where $\varphi$ has the form $\varphi(r) = C_s r^{-1-2s}$, with $s>\frac12$ and $C_s$ a constant.  The corresponding $(v,M)$ system, properly interpreted,  exhibits a regularizing effect and yields what we will refer to as $v$\textit{-strong} solutions.    
	
	\bigskip
	
	
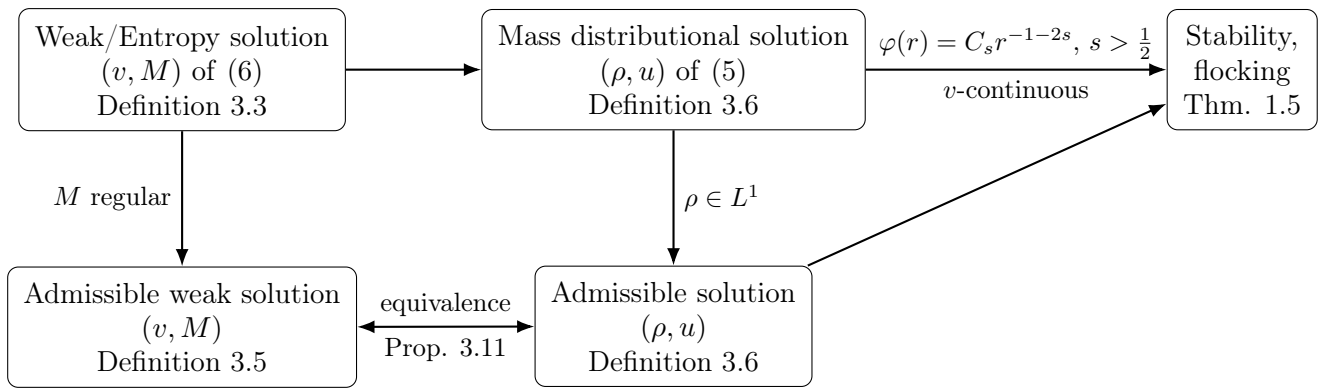
\begin{tikzpicture}[
		box/.style={draw, rounded corners, align=center, inner sep=6pt},
		arr/.style={-Latex, thick},
		darr/.style={Latex-Latex, thick},
		node distance=18mm and 18mm
		]
		
		\node[box] (VM)
		{Weak/Entropy solution\\
			$(v,M)$ of \eqref{e:reformulation}\\
			Definition \ref{def:vMweak}};
		
		\node[box, right=of VM] (RU)
		{Mass distributional solution\\
			$(\rho,u)$ of \eqref{e:EApuretopo}\\
			Definition \ref{def:MassDistributional}};
		
		\draw[arr] (VM) -- (RU);
		
		\node[box, below=of VM] (AW1)
		{Admissible weak solution\\
			$(v,M)$\\
			Definition \ref{def:admisr}};
		
		\node[box, below=of RU] (AW2)
		{Admissible solution\\
			$(\rho,u)$\\
			Definition \ref{def:MassDistributional}};
		
		\draw[arr] (VM) -- node[left, font=\small] {$M$ regular} (AW1);
		\draw[arr] (RU) -- node[right, font=\small] {$\rho\in L^1$} (AW2);
		
		\draw[darr] (AW1) -- node[above, font=\small]{equivalence} 
		node[below, font=\small] {Prop. \ref{p:equivalence}}
		(AW2);
		
		\node[box, right=40mm of RU] (SF)
		{Stability,\\
			flocking\\
			Thm. \ref{t:alignmentrough}};
		
		\draw[arr] (RU) -- node[above, font=\small] {$\varphi(r)=C_s r^{-1-2s}$, $s>\frac{1}{2}$ }
		node[below, font=\small] {$v$-continuous}
		(SF);
		\draw[arr] (AW2) -- (SF);
		
	\end{tikzpicture}
	
	\bigskip
	
	\subsubsection{Well-posedness and Focus on the $v$-equation}
	Our requirements on a `weak solution' of \eqref{e:reformulation} include the stipulation that $(m,t)\mapsto v(m,t)$ is essentially bounded, which guarantees that $A$ is Lipschitz in $m$, uniformly with respect to time.  The standard Kruzhkov theory \cite{Kruzkov1970} then guarantees the existence of a unique entropy solution $M$ of \eqref{e:reformulation}$_2$ associated to any initial condition in our classes of interest.  Therefore, a mass-distributional solution will exist on a time interval that is as long as is allowed by $v$, and this solution will be admissible for as long as $M$ remains absolutely continuous.  
	
	The hypotheses of Theorem \ref{t:classical}, suitably translated to the $(v,M)$ framework, provide a natural criterion for $M$ to remain absolutely continuous for all time.  On the other hand, a more comprehensive study of conditions guaranteeing absolute continuity of $M$ lie outside the scope of this paper.  Rather, we focus on understanding the $v$ equation, which will provide us with weak solutions of \eqref{e:reformulation} and  carries most information related to stability and long-time behavior.  We consider two main cases: bounded $\varphi$ and $\varphi(r) = C_s r^{-1-2s}$ for some $s\in (0,1)$ (and some constant $C_s$).  These two cases must be treated in different ways.  Indeed, let us denote the right side of \eqref{e:reformulation} by $-Lv$; that is, define
	\[
	(Lf)(m):=\int_0^1 \varphi(|m-m'|)\big(f(m) - f(m')\big) \dm'.
	\]
	If $\varphi$ is bounded, then $L$ is a bounded operator from $L^2(0,1)$ to itself, and $L^2$-valued solutions of \eqref{e:reformulation}$_1$ (which remain in $L^\infty$ if their initial data are in $L^\infty$) can be generated using standard theory of ODEs on a Hilbert space.  On the other hand, if $\varphi(r) = C_s r^{-1-2s}$, then the integral in the definition of $L$ should be interpreted in the principal value sense, and $L = \Lambda_{(0,1)}^s$ is an operator known as the \textit{regional fractional Laplacian} on $(0,1)$ (which makes \eqref{e:reformulation}$_1$ into a version of the fractional heat equation).  In fact, $\Lambda_{(0,1)}^s$ is unbounded, and one must prescribe its domain in order to complete the definition.  The natural choice for our purposes is a certain subspace of the Sobolev space $W^{s,2}(0,1)$, which makes $\Lambda^s_{(0,1)}$ into what is sometimes called the \textit{Neumann regional fractional Laplacian}; in what follows, we denote this operator simply by $\Lambda^s$. More details and further discussion are provided in Section~\ref{s:vsolutions}. 
	
	We compile the main takeaways regarding existence of solutions in the following Theorem; more complete statements will be provided below.
	\begin{theo}
		\label{t:existencerough}
		Consider the system \eqref{e:reformulation} with initial conditions $(v_0,M_0)$, where $v_0~\in~L^\infty(0,1)$ and $M_0$ is right-continuous and nondecreasing, with $\lim_{x\to -\infty} M_0(x)~=~0$ and $\lim_{x\to +\infty} M_0(x)~=~1$.  Assume that $\varphi$ is either bounded or takes the form $\varphi(r) = C_s r^{-1-2s}$,  for some $s\in (0,1)$.
		\begin{enumerate}[label = (\alph*)]
			\item If $\varphi$ is bounded, then there exists a unique weak solution  $(v,M)$ of \eqref{e:reformulation} in the sense of Definition~\ref{def:vMweak} associated to the initial data $(v_0, M_0)$.
			\item If $\varphi(r) = C_s r^{-1-2s}$ and the right side of \eqref{e:reformulation}$_1$ is interpreted as $-\Lambda^s v$, then there exists a unique $v$-strong solution $(v,M)$ of \eqref{e:reformulation} in the sense of Definitions \ref{def:strongheat} and \ref{def:Kruzkov} associated to the initial data $(v_0, M_0)$.
			\item The solutions in (a) and (b) both preserve the mean of $v$ and satisfy the natural energy equality:
			\begin{equation}
				\label{e:meanpreserved}
				\int_0^1 v(m,t) \dm = \int_0^1 v_0(m)\dm,   
			\end{equation}
			\begin{equation}
				\label{e:energyid}
				\frac{\dd}{\dt}\int_0^1 |v(m,t)|^2\dm + \int_{(0,1)\times (0,1)} \varphi(|m-m'|)|v(m,t) - v(m',t)|^2 \dm \dm' = 0.
			\end{equation}
			
		\end{enumerate}
		All of these solutions can automatically be interpreted as mass-distributional solutions of \eqref{e:EApuretopo},  which are admissible in the sense of Definition \ref{def:admist} if $(v,M)$ is admissible in the sense of Definition \ref{def:admisr}.
	\end{theo}
	The proof of (a) is mostly omitted, as is the part of (c) that deals with bounded kernels $\varphi$.  However, see Remark \ref{r:weakexist} for some notes on this case. We justify statement (b) by recalling a fundamental existence result for the regional fractional heat equation, see Theorem \ref{t:strongheat}. The other half of item (c) (dealing with singular rather than bounded kernels) is discussed in Remark \ref{r:heatenergy}.

	
	\begin{rem}\rm\label{rem:ddd}
		Suppose $\varphi$ is Lipschitz and nonincreasing. Then Theorem \ref{t:classical}, translated into new Lagrangian coordinates, essentially tells us that the mass-distributional solution generated from $(v_0, M_0)$ is admissible globally in time if the following quantity is nondecreasing:
		\begin{equation*}
			a(m) = v_0(m) + \int_0^1 \varphi(|m-m'|)(m-m') \dm'.
		\end{equation*} 
	\end{rem}

	\begin{rem}\rm\label{r:nowynew}
		In contrast to the case of Lipschitz kernels, it is unclear for singular $\varphi$ whether there is a clean way to upgrade from mass-distributional to admissible solutions simply by making additional assumptions on the initial data.  Part of difficulty is the lack of translation invariance of the regional fractional Laplacian operator, and furthermore that this operator does not commute with derivatives and is not amenable to difference-quotient analysis.  Relatedly, solutions $v$ of \eqref{e:reformulation}$_1$ (with right side $-\Lambda^s v$) will remain locally Lipschitz in $(0,1)$ but cannot be expected to be regular all the way to the boundary.  
		
		Despite the difficulties with generating admissible solutions under singular protocols, our notion of mass-distributional solutions is robust in this setting and continues to yield existence and uniqueness of solutions (though one must be careful with the interpretation of these solutions, c.f. Remark \ref{bonafide}), as well as meaningful flocking estimates (see Section \ref{s:LongTime}).
	\end{rem}
	
	\subsubsection{Stability, Uniqueness and Long-Time Behavior}	
	\begin{theo}\label{t:alignmentrough}~
		\begin{enumerate}[label = (\alph*)]
			\item Suppose that $\varphi$ is bounded and
			$$ c_\varphi:={\rm esssup}\left\{\int_0^1\frac{1}{\varphi(|m-m'|)}\dd m':\quad m\in (0,1) \right\}<+\infty,$$
			and let $(\rho_1, u_1)$ and $(\rho_2,u_2)$ be a pair of admissible solutions of \eqref{e:EApuretopo} with equal means. Then we have
			\begin{equation*}
				\int_{\R^2}|u_1(x,t)-u_2(y,t)|^2\dd\pi_t(x,y)\leq e^{-\frac{2t}{c_\varphi}}\int_{\R^2}|u_1(x,0)-u_2(y,0)|^2\dd\pi_0(x,y) 
			\end{equation*}
			where $\pi_t(x,y) := (M^{-1}_1(\cdot,t),M^{-1}_2(\cdot,t))_\sharp\lambda$, where $\lambda$ is the Lebesgue measure. In particular, the exponential alignment estimate
			$$ \int_\R|u_1(x,t)-\bar{u}_1|^2\dd\rho_1(x,t)\lesssim e^{-\frac{2t}{c_\varphi}}\xrightarrow{t\to+\infty} 0,\quad \bar{u}_1:=\int_\R u_1(x,0)\dd\rho_1(x,0)$$
			holds true.
			\item Let $(\rho_1, u_1)$ and $(\rho_2,u_2)$ be a pair of mass-distributional solutions of \eqref{e:EApuretopo} with equal means associated with $v$-strong solutions $(M_1,v_1)$ and $(M_2,v_2)$ of \eqref{e:reformulation}$_1$ under the protocol $\varphi(r) = C_sr^{-1-2s}$ with $s>\frac{1}{2}$, and the right side of \eqref{e:reformulation}$_1$ interpreted as $-\Lambda^s v$. Here $v_1,v_2$ are continuous representatives of their respective classes in $W^{s,2}(0,1)$. Then there exists a constant $\lambda_1>0$ such that for all $t\geq \tau>0$ we have
			$$ \sup_{x\in\R}|u_1(x,t)-u_2(x,t)|\lesssim e^{-\lambda_1 (t-\tau)}\|v_1(\tau)-v_2(\tau)\|_{W^{s,2}} $$
			In particular,
			$$\sup_{x\in\R}|u_1(x,t)-\bar{u}_1|\lesssim e^{-\lambda_1 t}\xrightarrow{t\to+\infty} 0, $$
			where
			$$ \bar{u}_1:=\int_\R u_1(x,0)\dd\rho_1(x,0). $$
			Moreover, if additionally $v_1(0),v_2(0)\in W^{s,2}(0,1)$ then we can take $\tau=0$ above.
			\item In particular, admissible solutions and $v$-strong mass-distributional solutions with initial data in  $W^{s,2}(0,1)$ are unique.
		\end{enumerate}
	\end{theo}
	Theorem \ref{t:alignmentrough} is proved in Section \ref{s:LongTime}, Proposition \ref{p:stabadmis} (a) and Proposition \ref{p:stabstrong} (b). Item (c) is an immediate consequence of (a) and (b).
	
	\begin{rem}\rm
		It is worthwhile to note that while weak solutions $(v,M)$ in Theorem \ref{t:existencerough}(a) are unique up to perturbation on Lebesgue-null sets, jumps of the cumulative distribution function $M$ produce non-uniqueness even in the $\rho$-a.e. sense. This is circumvented in Theorem \ref{t:alignmentrough}(a) by taking regular $M$ in admissible solutions and in (b) by taking the continuous representative $v$.    
	\end{rem}
	
	\section{Global-in-Time Classical Solutions for Regular Topological Protocols}
	
	\label{s:threshold}
	\noindent
	This section is dedicated to the discussion of classical solutions related to topological protocols in 1D. In particular, we shall prove Theorem \ref{t:classical}.
	
	The initial value problem for the system \eqref{e:EA}, written in Lagrangian coordinates (with reference domain $\mathbb{R}^d$ and reference density $\rho_0$) is given by 
	\begin{equation}
		\label{e:Lagrangian}
		\begin{cases}
			\dot{X}(\alpha,t) = V(\alpha, t), \\
			\dot{V}(\alpha,t) = -\int_{\mathbb{R}^d} \phi(X(\alpha, t), X(\gamma, t), t)(V(\alpha,t) - V(\gamma, t)) \rho_0(\gamma) \dd \gamma \\
			X(\alpha, 0) = \alpha, \quad V(\alpha, 0) = u_0(\alpha).
		\end{cases}
	\end{equation}
	Given a classical solution $(X,V)$ of \eqref{e:Lagrangian}, it is clear that one obtains a classical solution $(\rho, u)$ of \eqref{e:EA} through the relationships $\rho(X(\alpha, t),t) \det \nabla_\alpha X(\alpha, t) = \rho_0(\alpha)$ and $u(X(\alpha, t), t) = V(\alpha, t)$, on any time interval $[0,T)$ on which the following conditions are satisfied:
	\begin{enumerate}[label = (\alph*)]
		\item $\alpha\mapsto X(\alpha,t)$ remains invertible.
		\item $\alpha\mapsto \det \nabla_\alpha X(\alpha, t)$ remains bounded from below.
	\end{enumerate}

	When $d = 1$, conservation of mass implies that 
	\[
	\dd_{\rho(t)}(X(\alpha, t), X(\gamma, t)) = \dd_{\rho_0}(\alpha, \gamma),
	\]
	so that under the assumption \eqref{e:topokernel}, we have 
	\[
	\phi(X(\alpha, t), X(\gamma, t), t)
	= \varphi(\!\dd_{\rho_0}(\alpha, \gamma), X(\alpha, t) - X(\gamma, t)).
	\]
	Consequently, when $\varphi(\!\dd, z)$ is Lipschitz in $\dd$ and bounded in $z$, equation \eqref{e:Lagrangian}$_2$ can be rewritten (still under the assumption that $d = 1$) as 
	\[
	\frac{\dd}{\dt} \bigg[ V(\alpha, t) + \int_{\mathbb{R}} \Phi(\!\dd_{\rho_0}(\alpha, \gamma), X(\alpha, t) - X(\gamma, t)) \rho_0(\gamma) \dd\gamma\bigg] = 0,
	\]
	where $\Phi(\!\dd,z) = \int_0^z \varphi(\!\dd, \zeta)\dd\zeta$, as in \eqref{e:psi0def}.  Therefore, under the hypotheses of Theorem \ref{t:classical}, the system \eqref{e:Lagrangian} reduces to 
	\begin{equation}
		\label{e:1DLagrangian}
		\begin{cases}
			\dot{X}(\alpha, t) = \psi_0(\alpha) - \int_{\mathbb{R}} \Phi(\!\dd_{\rho_0}(\alpha, \gamma), X(\alpha, t) - X(\gamma, t)) \rho_0(\gamma) \dd\gamma, \\
			X(\alpha, 0) = \alpha.
		\end{cases}
	\end{equation}
	It is clear (e.g. from the Picard Theorem) that this system has a global-in-time classical solution $X(\alpha, t)$ under our working hypotheses.  We show that it translates to a global-in-time solution $(\rho, u)$ of the Euler Alignment system, whenever $\psi_0$ is nondecreasing.  It suffices to verify the conditions (a) and (b) on the flow map $X(\cdot, t)$ that were mentioned above.  We do both simultaneously by providing a lower bound on the difference quotient $\frac{X(\beta, t) - X(\alpha, t)}{\beta - \alpha}$, for $\alpha<\beta$.  We have
	\begin{align*} 
		\frac{\dd}{\dt} \big[ X(\beta,t) - X(\alpha,t) \big] 
		& = \psi_0(\beta) - \psi_0(\alpha) - \int_\R \bigg(\Phi\big( \!\dd_{\rho_0}(\beta, \gamma), X(\beta, t) - X(\gamma,t)\big) \\
		& \hspace{43 mm} - \Phi\big( \!\dd_{\rho_0}(\alpha, \gamma), X(\alpha, t) - X(\gamma,t)\big)\bigg) \rho_0(\gamma) \dd\gamma \\
		& = \psi_0(\beta) - \psi_0(\alpha) - \int_\mathbb{R} \int_{\alpha}^{\beta} \frac{\dd}{\dd\xi} \Phi(\!\dd_{\rho_0}(\xi,\gamma), X(\xi,t) - X(\gamma,t)) \dd\xi \,\rho_0(\gamma) \dd\gamma .
	\end{align*} 
	Furthermore,
	\begin{align*}
		\frac{\dd}{\dd\xi} \Phi(\!\dd_{\rho_0}(\xi,\gamma), X(\xi,t) - X(\gamma,t)) 
		& = \partial_1 \Phi(\!\dd_{\rho_0}(\xi, \gamma), X(\xi,t) - X(\gamma, t)) \cdot \frac{\dd}{\dd\xi} \bigg| \int_{\gamma}^\xi \rho_0 \bigg| \\
		& \qquad + \varphi(\!\dd_{\rho_0}(\xi, \gamma), X(\xi,t) - X(\gamma, t)) \partial_\xi X(\xi,t) \\[3ex]
		& = \bigg( \int_0^{X(\xi,t) - X(\gamma, t)} \partial_1 \varphi(\!\dd_{\rho_0}(\xi, \gamma), \zeta) \dd\zeta \bigg) \rho_0(\xi)\sgn(\xi - \gamma)  \\[1ex]
		&  \qquad + \varphi(\!\dd_{\rho(t)}(X(\xi,t), X(\gamma,t)), X(\xi,t) - X(\gamma, t)) \partial_\xi X(\xi,t).
	\end{align*}
	
	Since $\varphi$ is nonincreasing in its first component by assumption, it follows that the first term on the right is nonpositive, so that, putting everything together, we obtain 
	
	\begin{align*}
		\frac{\dd}{\dt} \big[ X(\beta,t) - X(\alpha,t) \big]  & = \psi_0(\beta) - \psi_0(\alpha) + \int_\mathbb{R} \int_{\alpha}^{\beta} \bigg| \int_0^{X(\xi,t) - X(\gamma,t)} \partial_1 \varphi(\dd_{\rho_0}(\xi,\gamma),\zeta) \dd \zeta \bigg| \rho_0(\xi)  \rho_0(\gamma) \dd\xi \dd \gamma \\
		& \qquad - \int_\mathbb{R} \int_{X(\alpha,t)}^{X(\beta,t)} \varphi(\!\dd_{\rho(t)}(y, X(\gamma,t)), y - X(\gamma,t))  \dy\,\rho_0(\gamma) \dd \gamma \\
		& \ge \psi_0(\beta) - \psi_0(\alpha) - \|\varphi\|_{L^\infty} \big[ X(\beta,t) - X(\alpha,t) \big],
	\end{align*} 
	where we have dropped the second term entirely.  Gr\"onwall's inequality then implies that 
	\[
	\frac{X(\beta, t) - X(\alpha, t)}{\beta - \alpha} \ge \exp(-\|\varphi\|_{L^\infty} t) + \frac{\psi_0(\beta) - \psi_0(\alpha)}{\beta - \alpha} \cdot \frac{ 1 - \exp(-\|\varphi\|_{L^\infty}t)}{\|\varphi\|_{L^\infty}}.
	\]
	This proves Theorem \ref{t:classical}.
	
	\begin{rem}\rm
		Defining
		\[
		\psi(x,t) = u(x,t) + \int_{\mathbb{R}} \Phi(\!\dd_{\rho(t)}(x,y), x-y) \rho(y,t)\dy, 
		\]
		we have
		\[
		\partial_t \psi + u \partial_x \psi = 0.
		\]    
		Consequently, $e = \partial_x \psi$ satisfies 
		\[
		\partial_t e + \partial_x (eu) = 0,
		\]
		provided that $\rho$ and $u$ are sufficiently regular.  The analysis of threshold conditions in Euler Alignment systems is often formulated in terms of $e$ (and sometimes $q = \frac{e}{\rho}$, which satisfies $\partial_t q + u \partial_x q = 0$) rather than $\psi$.  
	\end{rem}

	\begin{rem}\rm
		\label{r:classicalnegative}
		For protocols of the form $\phi = \phi(|x-y|)$, the second term in the expansion of $\frac{\dd}{\dt} \big[ X(\beta,t) - X(\alpha, t)\big]$ (which we dropped) is absent, so that $\psi_0(\beta) - \psi_0(\alpha)$ serves as an \textit{upper} bound on this time derivative.  If $\psi_0(\beta) - \psi_0(\alpha)$ is negative, then the solution must eventually lose regularity in this case, at or before time $T = \frac{\psi_0(\alpha) - \psi_0(\beta)}{\beta - \alpha}$.  This gives rise to a clean, critical threshold condition: regularity persists if $\psi_0$ is everywhere nondecreasing, but not if it decreases somewhere.  In the case where $\phi$ has a topological component, the sign-definite term that we dropped complicates the picture. A better understanding of this term would allow us to make more precise statements involving sufficient conditions for finite-time loss of regularity.  For instance, if we know \textit{a priori} that $\diam\supp(\rho(t))\le \cD$ for all $t$ in the time interval of regularity, then the term in question is bounded above by $\|\partial_1 \varphi\|_{L^\infty} \cD \dd_{\rho_0}(\alpha,\beta)$, and a sufficient condition for finite-time loss of regularity is 
		\[
		\psi_0(\beta) - \psi_0(\alpha) < - \|\partial_1 \varphi\|_{L^\infty} \cD \dd_{\rho_0}(\alpha, \beta),
		\]
		for some $\alpha, \beta\in \mathbb{R}$, with $\alpha<\beta$.
		Up to minor technical details, this condition can be reformulated as 
		\[
		q_0(\alpha) < - \|\partial_1 \varphi\|_{L^\infty} \cD,
		\]
		for some $\alpha\in \mathbb{R}$. 
	\end{rem}
	
	\section{Weak/Mass-Distributional, $v$-Strong, and Admissible Solutions}
	
	\label{s:Equivalence}
	
	In this section, we introduce and discuss three notions of solutions for each of the systems \eqref{e:EApuretopo} and~\eqref{e:reformulation}.  We start by defining weak, $v$-strong, and admissible solutions for \eqref{e:reformulation} in Section \ref{ss:vMsolutions}.  In Section \ref{ss:EAsolutions}, we introduce `mass-distributional' solutions $(\rho, u)$ as definitionally equivalent to weak solutions $(v,M)$ of~\eqref{e:reformulation}, and similarly for $v$-strong mass-distributional solutions.  Our definition of `admissible solutions' $(\rho, u)$ to~\eqref{e:EApuretopo}, on the other hand, is given directly in terms of the system \eqref{e:EApuretopo} itself rather than passing through the reformulation \eqref{e:reformulation}.  We prove in Section \ref{ss:equivalence} that there is a natural correspondence between the admissible solutions of \eqref{e:EApuretopo} and the admissible solutions of \eqref{e:reformulation}.
	
	\subsection{Solutions of the $(v,M)$ System}
	\label{ss:vMsolutions}
	
	\begin{defi}
		\label{def:vweak}
		We say that a function $v\in L^\infty((0,1)\times (0,T))$  is a weak solution of the equation \eqref{e:reformulation}$_1$ if for all $\eta\in C^\infty_c((0,1)\times (0,T))$, we have 
		\begin{equation}
			\label{e:weakref}
			\begin{split}
				0 = & \int_0^T\int_0^1\partial_t\eta\ v\dd m\dd t\\
				& +\frac{1}{2}\int_0^T\int_{(0,1)\times (0,1)} \varphi(|m-m'|)\big(v(m,t)-v(m',t)\big)\big(\eta(m,t)-\eta(m',t)\big)\dd m\dd m'\dd t.
			\end{split}
		\end{equation}
	\end{defi}
	Note the natural symmetrization here, which is superfluous if $\varphi$ is bounded but becomes crucial if $\varphi$ has a strong singularity at zero.
	
	\begin{defi}
		\label{def:Kruzkov}
		Given a measurable function $A:[0,1]\times (0,\infty)\to \mathbb{R}$ which is Lipschitz in its first argument, uniformly with respect to time, we say that a bounded, measurable function $M:\mathbb{R}\times (0,T)\to \mathbb{R}$ is an \textit{entropy solution} of the scalar conservation law $\partial_t M + \partial_x(A(M,t)) = 0$ (in the sense of Kruzhkov) if  the following inequality is satisfied in the sense of distributions. 
		\begin{equation} 
			\label{e:Kruzkov}
			\partial_t |M - m| + \partial_x\big[  \sgn(M - m)\big(A(M,t) - A(m,t)\big) \big] \le 0,
			\qquad \text{ for all } m\in \mathbb{R}.
		\end{equation} 
	\end{defi}
	
	\begin{defi}[Weak solution of \eqref{e:reformulation}]
		\label{def:vMweak}
		Assume $v_0\in L^\infty(0,1)$; let $M_0:\mathbb{R}\to \mathbb{R}$ be a right-continuous, nondecreasing function with $\lim_{x\to -\infty} M_0(x) = 0$ and $\lim_{x\to +\infty} M_0(x) = 1$. We say that $(v,M)$ is a weak solution of \eqref{e:reformulation} on the time interval $[0,T)$ and associated to the initial data $(v_0, M_0)$ if the following conditions are satisfied.
		\begin{enumerate}[label = (\roman*)]
			\item For a.e. $t\in (0,T)$, the function $x\mapsto M(x,t)$ is right-continuous and nondecreasing, with $\lim_{x\to-\infty}M(x,t)=0$ and $\lim_{x\to+\infty}M(x,t)=1$.  We also require $M\in C([0,T);L^1_{\loc}(\mathbb{R}))$ and $v\in L^\infty((0,1)\times(0,T))$.
			\item As $t\to 0+$, we have $v(\cdot, t)\to v_0$ in $L^2(0,1)$ and $M(\cdot, t)\to M_0$ in $L^1_{\loc}(\mathbb{R})$.
			\item The function $v$ is a weak solution of \eqref{e:reformulation}$_1$ in the sense of Definition \ref{def:vweak}.
			\item For $A:[0,1]\times (0,T)\to \mathbb{R}$ defined by $A(m,t) = \int_0^m v(m',t) \dm'$, we have that $M$ is an entropy solution of the scalar conservation law $\partial_t M + \partial_x(A(M,t)) = 0$ in the sense of Definition \ref{def:Kruzkov}.
		\end{enumerate}
	\end{defi}
	
	\begin{rem}[Existence of weak solutions]\label{r:weakexist}\rm
		Definition \ref{def:vMweak} is the natural (and in some sense minimal) way to define a solution of the initial value problem \eqref{e:reformulation} using Definitions \ref{def:vweak} and \ref{def:Kruzkov}. 
		\begin{itemize}
			\item Under assumptions (i) and (ii), with bounded communication $\varphi$, standard arguments ensure the existence of a unique strong solution in the class $C([0,T);L^2(0,1))$. This solution satisfies the energy equality and preserves the mean value, as required in item~(c) of Theorem~\ref{t:existencerough}. Moreover, any weak solution in the sense of Definition~\ref{def:vweak} that satisfies assumptions (i) and~(ii) necessarily coincides with this strong solution.
			\item The requirement that $v\in L^\infty$ guarantees that $A$ is Lipschitz in $m$, uniformly with respect to time.  Under the assumptions on $M_0$ and $A$ (through $v$), we are therefore always guaranteed the existence of a unique entropy solution of the scalar balance law in the sense of Definition \ref{def:Kruzkov}, by the standard Kruzhkov theory \cite{Kruzkov1970}, for as long as $v$ exists and remains in $L^\infty$.  From the point of view of the reformulated $(v,M)$ system, then, the remaining existence and uniqueness issues are entirely due to $v$.  
		\end{itemize}
		
		On the other hand, when we pass back to the level of the original variables $\rho$ and $u$, the cases where either $M$ is absolutely continuous in space or $v$ is regularized by a singular kernel $\varphi(r) = C_s r^{-1-2s}$, $s\in (0, 1)$ are especially favorable.  We therefore distinguish such solutions with their own terminology.
	\end{rem}
	
	\begin{defi}
		\label{def:admisr}
		Assume that $(v,M)$ is a weak solution of \eqref{e:reformulation} associated to the initial data $(v_0, M_0)$ in the sense of Definition \ref{def:vMweak}.
		\begin{enumerate}[label = (\alph*)]
			\item If $x\mapsto M(x,t)$ is absolutely continuous for a.e. $t\in (0,T)$, then we refer to $(v,M)$ as an `\textit{admissible} weak solution' associated to $(v_0, M_0)$.
			\item Alternatively, in the singular case $\varphi(r)=C_sr^{-1-2s}$, suppose that $v$ is a strong solution in the sense of Definition \ref{def:strongheat}.  Then we refer to $(v,M)$ as a `$v$-strong' solution.
		\end{enumerate} 
	\end{defi}
	
	Solutions that we call `$v$-strong' belong to the Sobolev space $W^{s,2}(0,1)$, which embeds in the space of continuous functions on $(0,1)$ if $s\in (\frac12, 1)$.  If a strong solution $v$ (in the sense of Definition \ref{def:strongheat}) is a.e. equal to a continuous function, we will always identify $v$  with its continuous representative. 
	
	\subsection{Solutions of the 1D Euler Alignment System}
	
	\label{ss:EAsolutions}
	
	\begin{defi}[Mass-distributional solution of \eqref{e:EApuretopo}]
		\label{def:MassDistributional}
		Let $\rho_0$ be a probability measure on $\mathbb{R}$ and assume $u_0\in L^\infty(\mathbb{R}, \rho_0)$.  We say that the pair $(\rho, u)$ is a `mass-distributional solution' of \eqref{e:EApuretopo} associated to the initial data $(\rho_0, u_0)$ if there exists a corresponding weak solution pair $(v,M)$ for the system~\eqref{e:reformulation} (in the sense of Definition \ref{def:vMweak}) that satisfies $\rho = \partial_x M$ (as distributions) and $u(x,t) = v(M(x,t),t)$ (pointwise), with initial conditions $(v_0, M_0)$ satisfying $\rho_0 = \partial_x M_0$ (as distributions) and $v_0 = u_0\circ M_0$ (pointwise). We say that a mass-distributional solution $(\rho, u)$ is `$v$-strong' if the corresponding $(v,M)$ is a $v$-strong solution in the sense of Definition \ref{def:admisr}.
	\end{defi}
	
	\begin{rem}\label{bonafide}\rm
		As previewed in the Introduction, the word `solution' in Definition \ref{def:MassDistributional} should be taken with a grain of salt, since the requirements of the definition are really imposed on the pair $(v,M)$ rather than on $(\rho, u)$.  Of course, if $\rho$ and $u$ happen to possess enough additional regularity, then the pair $(\rho, u)$ can be upgraded to a \textit{bona fide} solution in a more reasonable sense---for example, an `admissible' solution in the sense of Definition \ref{def:admist} below.  We also note, however, that interpreting the sense in which a general mass-distributional solution (with possibly measure-valued density) satisfies the momentum equation is likely to require a careful revisiting of the definition of $\dd_{\rho(t)}(x,y)$. It is not clear whether there is a choice that makes the distributional formulation of \eqref{e:EApuretopo}$_2$ compatible with the $v$ equation in this level of generality. 
	\end{rem}    
	
	\begin{defi}
		\label{def:admist}
		Given $(\rho_0, u_0)$ satisfying $\rho_0\in L^1(\mathbb{R})$, $\rho_0\ge 0$, $\int_{\mathbb{R}} \rho_0 = 1$, and $u_0\in L^\infty(\mathbb{R}, \rho_0)$, we say that the pair $(\rho, u)$ is an admissible weak solution of \eqref{e:EApuretopo} on the time interval $[0,T)$ and associated to the initial data $(\rho_0, u_0)$, if the following conditions are satisfied.
		\begin{enumerate}[label = (\roman*)]
			\item We have $\rho \in C([0,T); L^1(\R))$, with $\rho \ge 0$ and $\int_{\mathbb{R}} \rho(y,t) \dy = 1$ for a.e. $t\in (0,T)$, and $u$ is Borel measurable, with $\|u(\cdot, t)\|_{L^\infty(\mathbb{R}, \rho(\cdot, t))}\leq C$ for a.e. $t\in (0,T)$ and some constant $C$. 
			\item As $t\to 0+$, we have $\rho(t)\to \rho_0$, $\rho u(t)\to \rho_0 u_0$, and $(\rho|u|^2)(t)\to \rho_0 |u_0|^2$ in the sense of distributions.
			\item The continuity equation $\partial_t \rho + \partial_x (\rho u) = 0$ is satisfied on $\mathbb{R} \times (0,T)$ in the sense of distributions.  
			\item  The momentum equation \eqref{e:EA}$_2$ is satisfied in the following weak sense: For all $\eta\in C_c^\infty((0,1)\times(0,T))$ and $\psi(x,t):=\eta(M(x,t),t)$ we require
			\begin{equation}
				\label{e:weaktopol}
				\begin{split}
					0 = & \int_0^T\int_\R(\partial_t\psi  +u\partial_x\psi)\rho u \dd x\dd t \\
					& + \frac{1}{2}\int_0^T\int_{\R\times \mathbb{R}}\varphi(\!\dd_{\rho(t)}(x,y))\big(u(x,t)-u(y,t)\big)\big(\psi(x,t)-\psi(y,t)\big)\rho(x,t)\rho(y,t)\dd x\dd y\dd t.
				\end{split}
			\end{equation}
			
			\noindent
			(We require in particular that both sides of \eqref{e:weaktopol} must be well-defined and finite for each $\psi$ of the specified form.)
		\end{enumerate}
	\end{defi}
	
	\begin{rem}\rm
		An approximation argument shows that if \eqref{e:EA}$_2$ is satisfied in the sense of distributions, then item (iv) in Definition \ref{def:admist} is also satisfied.  The class of test functions used here has been tailored to be compatible with the reformulated system \eqref{e:reformulation}.
	\end{rem}
	
	\begin{rem}\rm
		By definition, mass-distributional solutions $(\rho, u)$ of \eqref{e:EApuretopo} are in natural correspondence with weak solutions $(v,M)$ of \eqref{e:reformulation}.   According to the proposition in the next subsection, admissible solutions $(\rho, u)$ of \eqref{e:EApuretopo} are in natural correspondence with admissible solutions of \eqref{e:reformulation}.  It follows that a mass-distributional solution $(\rho, u)$ of \eqref{e:EApuretopo} is in fact admissible if $\rho(\cdot, t)\in L^1(\mathbb{R})$ for a.e. $t\in (0,T)$, while any admissible solution of \eqref{e:EApuretopo} is also a mass-distributional solution.
	\end{rem}
	
	\subsection{Equivalence for Admissible Solutions}
	
	\label{ss:equivalence}
	
	\begin{prop} 
		\label{p:equivalence}
		The two systems  \eqref{e:EApuretopo} and \eqref{e:reformulation} are `equivalent for admissible weak solutions' in the following sense.
		\begin{enumerate}[label = (\alph*)]
			\item Suppose $(\rho, u)$ is an admissible weak solution of \eqref{e:EApuretopo} in the sense of Definition \ref{def:admist}, on the interval $[0,T)$ and associated to the initial data $(\rho_0, u_0)$.  Define $M(\cdot, t)$ and its (left-continuous) generalized inverse $M^{-1}(\cdot, t)$ via 
			\[
			M(x,t) = \int_{-\infty}^x \rho(y,t)\dy,\qquad 
			M^{-1}(m,t) = \inf\{ x\in \mathbb{R}: M(x,t)\ge m\}.
			\]
			Define $M_0$ and $M_0^{-1}$ similarly in terms of $\rho_0$.  Finally, define $v$ and $v_0$ via 
			\[ 
			v(m,t) = u(M^{-1}(m,t),t),
			\qquad 
			v_0(m) = u_0(M_0^{-1}(m)).
			\]
			Then $(M,v)$ is an admissible weak solution of \eqref{e:reformulation} on $[0,T)$ in the sense of Definition \ref{def:admisr} associated to the initial conditions $(M_0, v_0)$.
			\item Conversely, suppose that $(v,M)$ is an admissible weak solution of \eqref{e:reformulation} on $[0,T)$ in the sense of Definition \ref{def:admisr}, with initial conditions $(M_0, v_0)$. Define $\rho$, $u$, $\rho_0$, and $u_0$ via  
			\[
			\rho(x,t) = \partial_x M(x,t),
			\qquad 
			\rho_0(x) = \partial_x M_0(x),
			\qquad
			u(x,t) = v(M(x,t),t),
			\qquad 
			u_0(x) = v_0(M_0(x)).
			\]
			Then $(\rho, u)$ is a weak solution of \eqref{e:EApuretopo} in the sense of Definition \ref{def:admist}, associated to the initial conditions $(\rho_0, u_0)$.
		\end{enumerate}
	\end{prop}
	\begin{rem}\rm
		The above Proposition fashions associations $(\rho, u)\mapsto (M,v)$ and $(\widetilde{M},\widetilde{v})\mapsto (\widetilde{\rho}, \widetilde{u})$.  Proposition \ref{p:CDFprops}(d) guarantees that these two associations are inverses of one another. This is trivial for $\rho \leftrightarrow M$ but not completely obvious at the level of the velocities. 
	\end{rem}

	\begin{proof}[Proof of Proposition \ref{p:equivalence}]
		(a) Under the stated assumptions, it is standard that requirement (i) from Definition \ref{def:admist} implies requirement (i) from Definition \ref{def:vMweak}, and the validity of the continuity equation $\partial_t\rho + \partial_x(\rho u) = 0$ in the sense of distributions is equivalent to that of the transport equation $\partial_t M + u\partial_x M =~0$.  The definition of $v$ from (a) guarantees that $A(M(x,t),t) = \int_{(-\infty, x]} \rho u(x,t)\dx$ (cf. Proposition \ref{p:CDFprops}(e)), and thus that $\partial_x(A(M,t)) = \rho u(\cdot, t)$ in the sense of distributions.  As already noted in Remark \ref{r:weakexist}, the regularity of $M$ then guarantees that $M$ is the entropy solution of the scalar conservation law \eqref{e:reformulation}$_2$.
		
		To show that \(v(t)\to v_0\) in \(L^2(0,1)\) as \(t\to 0+\), we observe that the distributional convergence stated in item~(ii) of Definition~\ref{def:admist} can be strengthened to allow testing against constant functions. Indeed, by item~(i), the density \(\rho\) is continuous in \(L^1\) and the velocity \(u\) is uniformly bounded. These properties imply uniform integrability of the family \(\{x \mapsto |u(x,t)|^2 \rho(x,t)\}_{t\in[0,T)}\).
		Thus
		\[
		\|v(t)\|_{L^2(0,1)}^2 = \int_{\mathbb{R}} (\rho |u|^2)(t)\dx \to \int_{\mathbb{R}} \rho_0 |u_0|^2 \dx = \|v_0\|_{L^2(0,1)}^2,
		\]
		so it suffices to prove that $v(t)$ converges \textit{weakly} to $v_0$ in $L^2(0,1)$.  To this end, choose $\eta\in C^\infty_c((0,1))$.  We have
		\begin{align*}
			\bigg| \int_0^1 \eta(m)v(m,t) \dm - \int_0^1 \eta(m) v_0(m)\dm\bigg|  
			& \le \bigg| \int_{\mathbb{R}} \big[ \eta(M(x,t)) - \eta(M_0(x))\big] \rho u(x,t)\dx \bigg| \\
			& \qquad + 
			\bigg| \int_{\mathbb{R}} \eta(M_0(x)) \rho u(x,t) \dx - \int_{\mathbb{R}} \eta(M_0(x)) \rho_0 u_0(x) \dx  \bigg|.
		\end{align*}
		The second term on the right tends to zero as $t\to 0+$, by the assumed convergence of $\rho u(t)$ to $\rho_0 u_0$.  For the first term, we note that since $M(t)\to M_0$ in $L^1_{\loc}(\mathbb{R})$ and $v\in L^\infty$, it follows that $\{\rho u(t)\}_{t\in (0,T)}$ is uniformly integrable and $\eta\circ M(t)\to \eta\circ M_0$ in measure.  Combining these two facts (and remembering that $\eta$ is bounded) allows us to conclude that the first term on the right side of the inequality above also tends to zero as $t\to 0+$.  This establishes (ii).
		
		Once we have $\partial_t M + u \partial_x M = 0$, with $u\partial_x M\in L^1_{\loc}$, it follows that
		$\partial_t M\in L^1_{\loc}$ and thus $M\in W^{1,1}_{\loc}(\R\times(0,T))$.
		Therefore the Sobolev chain rule applies:
		$\partial_t(\eta(M,t))=\partial_m\eta(M,t)\partial_t M+\partial_t\eta(M,t)$ and
		$\partial_x(\eta(M,t))=\partial_m\eta(M,t)\partial_x M$, which yields
		\begin{equation}
			\label{e:etachain}
			\partial_t \big( \eta(M(x,t),t) \big) + u(x,t) \partial_x (\eta(M(x,t),t)) = \partial_t \eta(M(x,t),t),
		\end{equation} 
		in the sense of distributions.  Thus for $\psi(x,t) = \eta(M(x,t),t)$, we have 
		\[
		\int_0^T \int_0^1 \partial_t\eta\, v(m,t)\dm \dt
		= \int_0^T \int_{\mathbb{R}} \partial_t \eta(M(x,t),t) \rho u(x,t) \dx \dt
		= \int_0^T \int_{\mathbb{R}} (\partial_t \psi + u \partial_x \psi)(x,t) \rho u(x,t) \dx \dt,
		\]
		where in the first equality, we have used the fact that $M(t)\circ M^{-1}(t)$ is, for a.e. $t$, the identity map on $(0,1)$, since $x\mapsto M(x,t)$ is continuous for a.e. $t$.  The right hand side of the above is equal to $-\frac12$ times the integral below, which can be manipulated as follows:    
		\begin{align*} 
			& \int_0^T \int_{\mathbb{R}\times \mathbb{R}} \phi(x,y,t)\big(u(x,t) - u(y,t)\big)\big(\psi(x,t) - \psi(y,t) \big) \rho(x,t) \rho(y,t) \dx \dy \dt \\
			& = \int_0^T \int_{\mathbb{R}\times \mathbb{R}} \varphi(|M(x,t)-M(y,t)|) \big(u(x,t) - u(y,t) \big) \big( \eta(M(x,t),t) - \eta(M(y,t),t)\big) \rho(x,t) \rho(y,t) \dx \dy \dt \\
			& = \int_0^T \int_{(0,1)\times (0,1)} \varphi(|m-m'|) \big(u(M^{-1}(m,t),t) - u(M^{-1}(m',t),t) \big) \big( \eta(m,t) - \eta(m',t)\big) \dm \dm' \dt\\
			& = \int_0^T \int_{(0,1)\times (0,1)} \varphi(|m-m'|) \big(v(m,t) - v(m',t) \big) \big( \eta(m,t) - \eta(m',t)\big) \dm \dm' \dt.
		\end{align*}
		We conclude that $(v,M)$ satisfies the requirements of Definitions \ref{def:vMweak} and \ref{def:admisr}.
		
		\medskip
		(b) The requirement that $\rho\in C([0,T);L^1(\mathbb{R}))$ follows easily from the regularity requirements on $M$, and the continuity equation $\partial_t \rho + \partial_x (\rho u) = 0$ can be obtained by differentiating $\partial_tM + u\rho = 0$, which in turn follows from the distributional identity $\partial_x(A(M,t)) = u\rho$.  
		It is clear that $\rho(t)\to \rho_0$ in the sense of distributions; for $\rho u$ and $\rho |u|^2$, we argue as follows.  The fact that $M(t)\to M_0$ in $L^1_{\loc}(\mathbb{R})$ implies that $M^{-1}(t)\to M_0^{-1}$ in measure; therefore, $\eta\circ M^{-1}(t)\to \eta\circ M_0^{-1}$ in measure for any $\eta\in C^\infty_c(\mathbb{R})$.  Since $\eta$ is bounded, the family $\{\big|\eta\circ M^{-1}(t)\big|^2\}_{t\in (0,T)}$ is trivially uniformly integrable, so  Vitali's Convergence Theorem then guarantees that $\eta\circ M^{-1}(t)\to \eta\circ M_0^{-1}$ in $L^2((0,1))$ as $t\to 0+$.  Thus, 
		\[
		\int_{\mathbb{R}} \rho u(x,t) \eta(x) \dx = 
		\int_0^1 v(m,t) \eta(M^{-1}(m,t)) \dm \stackrel{t\to 0+}{\longrightarrow} \int_0^1 v_0(m) \eta(M_0^{-1}(m)) \dm
		= \int_{\mathbb{R}} \rho_0 u_0(x) \eta(x) \dx.
		\]
		A similar argument (using $L^1$ convergence of $|v(t)|^2$ to $|v_0|^2$) can be used to show that 
		\[
		\int_{\mathbb{R}} \rho |u|^2(x,t) \eta(x) \dx \stackrel{t\to 0+}{\longrightarrow}  \int_{\mathbb{R}} \rho_0 |u_0|^2(x) \eta(x) \dx.
		\]
		As for the momentum equation, we essentially reverse the steps from the proof of (a).   Given $\eta\in C_c^\infty((0,1)\times (0,T))$ and $\psi(x,t) = \eta(M(x,t),t)$, we have (using \eqref{e:etachain} and $M(t)_{\sharp} \rho(t) = \lambda$, by Proposition~\ref{p:CDFprops}(c)) that 
		\begin{small} 
			\begin{align*} 
				& \int_0^T \int_{\mathbb{R}} (\partial_t \psi + u \partial_x \psi) \rho u(x,t)\dx \dt 
				= \int_0^T \int_{\mathbb{R}} \partial_t \eta(M(x,t),t) v(M(x,t),t) \rho(x,t) \dx \dt \\
				& = \int_0^T \int_0^1 \partial_t \eta(m,t) v(m,t) \dm \dt = -\frac12\int_0^T \int_{(0,1)\times (0,1)} \varphi(|m-m'|) \big(v(m,t) - v(m',t) \big) \big( \eta(m,t) - \eta(m',t)\big) \dm \dm' \dt
				\\
				& = -\frac12\int_0^T \int_{\mathbb{R}\times \mathbb{R}} \varphi(|M(x,t) - M(y,t)|) \big( v(M(x,t),t) - v(M(y,t),t) \big) \big( \eta(M(x,t),t) - \eta(M(y,t),t) \big) \rho(x,t) \rho(y,t) \dx \dy \dt \\
				& = -\frac12\int_0^T \int_{\mathbb{R}\times \mathbb{R}} \phi(x,y,t)\big(u(x,t) - u(y,t)\big)\big(\psi(x,t) - \psi(y,t) \big) \rho(x,t) \rho(y,t) \dx \dy \dt.
			\end{align*} 
		\end{small} 
		Thus the requirements of Definition \ref{def:admist} are satisfied.    
	\end{proof}
	
	\section{The $v$ Equation with Singular Kernels:  Existence of Mass-Distributional Solutions}
	
	\label{s:vsolutions}
	In this section, we single out the special case where 
	\[
	\varphi(r) = C_s r^{-1-2s},
	\qquad r>0,
	\]
	for a fixed $s\in (0,1)$ and a corresponding normalizing constant $C_s$.  (The value of $C_s$ is basically inconsequential for our purposes, but for the sake of following standard conventions, we specify $C_s = \frac{s 4^s}{\sqrt{\pi}} \cdot \frac{\Gamma(\frac12 + s)}{\Gamma(1-s)}$.)  In this case, the integral on the right side of \eqref{e:reformulation}$_1$ can be understood in a principal value sense, and the entire system can be written as 
	\begin{equation}
		\label{eq:heatncons}
		\begin{cases}
			\p_t v = -\Lambda_{(0,1)}^s v \\
			\p_t M + \p_x(A(M,t)) = 0.
		\end{cases}
	\end{equation}
	
	In \eqref{eq:heatncons}$_1$, the operator $\Lambda_{(0,1)}^s$ is the \textit{regional fractional Laplacian} operator.  In general, the regional fractional Laplace operator $\Lambda_\Omega^s$ associated to an open, bounded subset $\Omega$ of $\mathbb{R}^d$ is given by 
	\[ 
	\Lambda^s_\Omega \xi (z) := C_{d,s} {\rm p.v.}\int_{\Omega}\frac{\xi(z)-\xi(z')}{|z-z'|^{d+2s}}\dd z' := \lim_{\epsilon\searrow 0} \int_{\Omega\cap\{|z-z'|>\epsilon\}} \frac{\xi(z)-\xi(z')}{|z-z'|^{d+2s}}\dd z',\quad z\in\Omega,
	\]
	for sufficiently regular $\xi:\O\to \mathbb{R}$.  Here $C_{d,s}$ is a normalizing constant, and $C_{1,s} = C_s$ in the notation introduced above. For more on the regional fractional Laplacian, we refer the reader to \cite{GalWarma2016}, \cite{SobolevHitchhiker}, and references therein. In what follows, we restrict attention to the case where $d = 1$ and $\Omega = (0,1)$, and we introduce only the tools we need for our analysis.  
	\subsection{The Regional Fractional Laplacian Operator}
	As mentioned in Section \ref{ss:SOR}, the domain we will use for $\Lambda^s$ will be a subspace of the  \textit{fractional Sobolev space} $(W^{s,2}(0,1), \|\cdot\|_{W^{s,2}(0,1)})$.  The latter is given by
	\[
	W^{s,2}(0,1) := \left\{v\in L^2(0,1):\quad \int_{(0,1)\times (0,1)} \frac{|v(m)-v(m')|^2}{|m-m'|^{1+2s}}\dm \dm'<\infty\right\}, 
	\]
	\[
	\|v\|_{W^{s,2}(0,1)}:= \left(\int_{(0,1)} |v|^2\dd x + \frac{C_{s}}{2} \int_{(0,1)\times (0,1)} \frac{|v(m)-v(m')|^2}{|m-m'|^{1+2s}}\dm \dm' \right)^\frac{1}{2}. 
	\]

	
	
	Given a fixed $s\in (0,1)$, we define the bilinear form $\cE^s$ via 
	\[
	\mathcal{E}^s(v,w) = \frac{C_{s}}{2}\int_{(0,1)\times (0,1)} \frac{(v(m)-v(m'))\cdot(w(m)-w(m'))}{|m-m'|^{1+2s}}\dm\dm',
	\qquad 
	v,w\in W^{s,2}(0,1).
	\]
	The definition of $\cE^s$ naturally gives rise to the corresponding definition of the domain $D(\Lambda^s)$ of the closed linear selfadjoint operator $\Lambda^s$, which we refer to as the \textit{Neumann regional fractional Laplacian} operator:
	\begin{equation}
		D(\Lambda^s) := \{v\in W^{s,2}(0,1)  :\ \exists f\in L^2(0,1) \text{ s.t. } \forall w\in W^{s,2}(0,1) \ \mbox{ we have }\ {\mathcal E}^s(v,w) = (f,w)_{L^2(0,1)}\}. 
	\end{equation}
	Finally, for $v\in D(\Lambda^s)$, we define $\Lambda^s(v)$ to be the unique element $f$ of $L^2(0,1)$ such that $\cE^s(v,w) = (f,w)_{L^2(0,1)}$ for all $w$ belonging to $W^{s,2}(0,1)$.  
	
	A more explicit interpretation of $D(\Lambda^s)$ for $s\in (\frac12, 1)$ is presented in  \cite{GalWarma2016}. We will also describe a few special considerations below in Section \ref{ss:vinit}.
	
	\subsection{Strong Solutions of the Fractional Heat Equation}
	Consider the following initial value problem:      
	\begin{equation}
		\label{e:heat}
		\begin{cases} 
			\partial_t v = -\Lambda^s v,
			\qquad (m,t)\in (0,1)\times (0,\infty) \\
			v(0) = v_0
		\end{cases} 
	\end{equation}
	We now record the definitions of weak and strong solutions of \eqref{e:heat}, and we recall from \cite{GalWarma2016} that strong solutions of \eqref{e:heat} exist and are unique for $L^\infty$ initial data.  (We document this statement as our Theorem~\ref{t:strongheat}.) Since our theory of \eqref{e:reformulation} requires $v\in L^\infty$, we are really only interested in strong solutions; however, we include the definition of a weak solution as well for the sake of completeness.  In what follows, we use the notation $W^{-s,2}(0,1) := (W^{s,2}(0,1))^*$ to refer to the topological dual space of $W^{s,2}(0,1)$, and $\langle\cdot,\cdot\rangle$ will denote the corresponding duality pairing.  
	
	\begin{defi}[Weak solution of the fractional heat equation]\label{def:weakheat}
		The function $v$ is a {\rm weak solution} of \eqref{e:heat} on the time interval $(0,T)$ and associated to the initial condition $v_0\in L^2(0,1)$ if for a.e. $t\in(0,T)$ the following properties hold.
		
		\begin{enumerate}[label = (\roman*)]
			\item The function $v$ and its time derivative $\partial_t v$ satisfy
			\begin{align*}
				v &\in L^\infty(0,T;L^2(0,1))\cap L^2(0,T; W^{s,2}(0,1)),\\
				\partial_t v &\in L^2(0,T; W^{-s,2}(0,1)).
			\end{align*}
			\item For all $\xi\in W^{s,2}(0,1)$ and for a.e. $t\in (0,T)$ we have
			\begin{equation}
				\label{e:fractionalheatweak}
				\langle\partial_t v(t),\xi\rangle + {\mathcal E}^s(v(t),\xi) = 0 
			\end{equation} 
			and $v(0)=v_0$ as elements of $L^2(0,1)$.
			\item Finally, the solution satisfies the energy equality
			\[
			\frac{1}{2}\|v(t)\|_{L^2(0,1)}^2 + \int_0^t{\mathcal E}^s(v(\tau),v(\tau))\dd \tau = \frac{1}{2}\|v_0\|_{L^2(0,1)}^2. 
			\]
		\end{enumerate}
	\end{defi}
	
	\begin{defi}[Strong solution of the fractional heat equation]\label{def:strongheat}
		Let $v_0\in L^\infty(0,1)$. A weak solution $v$ satisfying Definition \ref{def:weakheat} is {\rm strong} if
		\[
		v\in W^{1,\infty}_{loc}((0,T); L^2(0,1))\cap C((0,T);L^\infty(0,1)) 
		\]
		and $v(t)\in D(\Lambda^s)$ for a.e $t\in (0,T)$.
	\end{defi}
	
	We can now state the existence result from \cite{GalWarma2016}) that we mentioned above.
	
	\begin{theo}
		\label{t:strongheat}
		For any $v_0\in L^\infty(0,1)$ and any $T>0$, there exists a unique strong solution (in the sense of Definition \ref{def:strongheat}) of the fractional heat equation \eqref{e:heat} associated to the initial data $v_0$ on the time interval $[0,T)$.
	\end{theo}
	
	\begin{rem}\rm\label{r:heatenergy}
		The spaces $W^{s,2}(0,1)\subset L^2(0,1)\subset W^{-s,2}(0,1)$ form a Gelfand triple and thus, testing \eqref{e:fractionalheatweak} by $\xi = v(t)$ for a.e. $t>0$, and applying the Lions-Magenes lemma, we obtain the differential energy equality
		\[
		\frac{1}{2}\frac{d}{dt}\|v(t)\|_{L^2(0,1)}^2=-{\mathcal E}^s(v(t),v(t)),\quad \text{for a.e. } t>0. 
		\]
		A similar argument leads to the conclusion that the mean $\bar{v}(t)=\int_0^1v(m,t)\dd m$ is a {\it constant-in-time} solution in the sense of Definition \ref{def:strongheat}.
	\end{rem}

	\subsection{Interpretation of $D(\Lambda^s)$}
	
	\label{ss:vinit}
	
	In Section \ref{s:LongTime} below, when we treat the protocol $\varphi(r) = C_s r^{-1-2s}$, we restrict attention to solutions $v$ generated from \eqref{e:heat} equipped with the \textit{Neumann} regional fractional Laplacian with $s>\frac12$.  We note, however, that our results could in principle also be extended to apply to \eqref{e:reformulation} equipped with the \textit{Dirichlet regional fractional Laplacian}.  (The Theorem we rely on from \cite{GalWarma2016}, on strong solutions of the fractional heat equation, is stated for both operators.)  We point out a few differences between these operators and explain why we have chosen to focus on the Neumann case.  For the rest of this section only, we denote the Dirichlet regional fractional Laplacian by  $\Lambda_{\mathsf{D}}$ and the Neumann regional fractional Laplacian by $\Lambda_{\mathsf{N}}$.  
	
	The operator $\Lambda^s_{\mathsf{D}}$ is defined similarly to $\Lambda^s_{\mathsf{N}}$, except that instead of using $W^s(0,1)$ for the domain of the associated bilinear form, one uses $W_0^s(0,1)$ (i.e., the closure of $C_c^\infty(0,1)$ in $W^s(0,1)$), and as a result (c.f. \cite{GalWarma2016}), we have $D(\Lambda^s_{\mathsf{D}})\subset W_0^s(0,1)$.  For $s\in (\frac12, 1)$, the space $W_0^s(0,1)$ enforces zero boundary conditions.  Consequently, replacing $\Lambda_{\mathsf{N}}$ with $\Lambda_{\mathsf{D}}$ in \eqref{e:heat} would enforce $v(0+) = v(1-) = 0$ for a strong solution $v$.  This has consequences for the associated solution $(\rho, u)$ of the Euler Alignment system.  If $\rho$ is compactly supported, then this would force $u$ to be zero at the extremes of $\supp \rho$, thus fixing the location of these extremes for all time: $\max \supp \rho(t) = \max \supp \rho_0$ and $\min \supp \rho(t) = \min \supp \rho_0$.  Confinement of the support of $\rho(t)$ is a phenomenon of interest (in fact, it is sometimes already associated with the terminology `flocking'), but from a modeling perspective, we would like to see this confinement arise as a result of the alignment forcing---not the boundary conditions.
	
	In contrast to $\Lambda_{\mathsf{D}}$, the operator $\Lambda_{\mathsf{N}}$ does not impose restrictions on the boundary values of $v$.  Gal and Warma's analysis does provide some insight into the sense in which a `fractional Neumann' boundary condition is enforced; however, it is much less rigid than the requirements imposed by $\Lambda^s_{\mathsf{D}}$.  Rather than getting into the technical details here, we give a brief example.  
	Suppose that $I(t):=\{x:\rho(x,t)>0\}$ is an interval at time zero, and thus for all time.  Assume that inside $I(t)$, we have that $\rho(t)$ is bounded below and continuous.  (This forces a jump in $\rho(t)$ at each of the endpoints of $I(t)$.) Then the fractional Neumann condition mentioned above are satisfied whenever $u(t)$ and $\partial_x u(t)$ are continuous and bounded inside $I(t)$, for instance.  Interpreting this condition for less regular solutions is more subtle, but using $\Lambda_{\mathsf{N}}$ remains a better choice for our purposes than using $\Lambda_{\mathsf{D}}$.

	\section{Stability and Long--Time Behavior}
	\label{s:LongTime}
	
	The goal of this section is to prove Theorem \ref{t:alignmentrough}; we study stability and long--time alignment of \eqref{e:EA} in two cases: for admissible and for $v$-strong mass-distributional solutions. Two key ingredients of our arguments are the energy identity
	\begin{equation}\label{e:energyid2}
		\frac{1}{2}\frac{d}{dt}\|v(t)\|^2_{L^2}=-{\mathcal E}_\varphi(v(t),v(t)),\quad {\mathcal E}_\varphi(v,v) = \int_{(0,1)\times(0,1)}\varphi(|m-m'|)|v(m)-v(m')|^2\dd m\dd m',
	\end{equation}
	and the conservation of the spatial mean
	$$ \int_0^1v(m,t)\dd m = const.$$
	Both of the above properties are satisfied for admissible and $v$-strong solutions by Theorem \ref{t:existencerough}.
	We begin by proving the following Poincar\' e-type inequality.
	\begin{lem}\label{l:poinc}
		If
		$$ c_\varphi:={\rm esssup}\left\{\int_0^1\frac{1}{\varphi(|m-m'|)}\dd m':\quad m\in (0,1) \right\}<+\infty,$$
		then for $\bar{v}:=\int_0^1v(m)\dd m$ we have
		$$ \int_0^1|v(m)-\bar{v}|^2\dd m\leq c_\varphi {\mathcal E}_\varphi(v,v). $$
	\end{lem}
	\begin{proof}
		For a.e. $m\in (0,1)$ we have $v(m)-\bar{v} = \int_0^1v(m) - v(m')\dd m'$ and thus
		\begin{align*}
			|v(m) - \bar{v}|^2 &\leq \left|\int_0^1\sqrt{\varphi(|m-m'|)}(v(m)-v(m'))\ \frac{1}{\sqrt{\varphi(|m-m'|)}}\dd m'\right|^2\\
			&\leq \int_0^1\frac{1}{\varphi(|m-m'|)}\dd m'\int_0^1\varphi(|m-m'|)|v(m)-v(m')|^2\dd m'.
		\end{align*} 
		Taking the supremum in the former integral on the right-hand side above and integrating the latter with respect to $m$ yields the desired inequality.
	\end{proof}
	
	\begin{rem}\rm
		The constant $c_\varphi$ is finite in many natural cases that include both singular and regular protocols. For instance:
		
		\begin{enumerate}[label = \arabic*.]
			\item If $\varphi(r)\ge \varepsilon>0$ on $[0,1]$, then
			\[
			\int_0^1 \frac{1}{\varphi(|m-m'|)}\dd m' \le \varepsilon^{-1},
			\]
			hence $c_\varphi<\infty$.
			\item Condition $c_\varphi<\infty$ does not require $\varphi$ to be
			bounded away from zero, Indeed, if $\varphi(r)=(1-r)^\alpha$ with $\alpha\in(0,1)$, then
			$\varphi(1)=0$, yet
			\[
			\int_0^1 \frac{1}{(1-|m-m'|)^\alpha}\dd m'<\infty
			\]
			for every $m\in[0,1]$ (the only singularity is integrable for
			$\alpha<1$). Thus $c_\varphi<\infty$ even though $\varphi$ vanishes at $r=1$.
			\item If $\varphi=0$ on a set of positive measure, then
			$1/\varphi=\infty$ on that set and $c_\varphi=+\infty$; in particular,
			short-range kernels in the mass variable are excluded here.
			\item In the singular case of $\varphi(r)=C_s r^{-1-2s}$ with $s\in(0,1)$ we have
			\[
			c_\varphi\leq 2C_s \int_0^1|m'|^{1+2s}\dd m'<\infty.
			\]
		\end{enumerate}
	\end{rem}
	
	\begin{lem}\label{l:alignment}
		Let $v_1$ and $v_2$ be two weak or $v$-strong solutions of \eqref{e:reformulation}$_1$ with equal means. Then the following stability estimate holds true:
		\begin{equation}\label{e:l-alignment}
			\|v_1(t)-v_2(t)\|_{L^2}\leq e^{-\frac{t}{c_\varphi}}\|v_1(0)-v_2(0)\|_{L^2}
		\end{equation}
	\end{lem}
	\begin{proof}
		By linearity, if $v_1$ and $v_2$ are solutions to \eqref{e:reformulation}$_1$ (in respective classes of regularity), then so is $\omega:=v_1-v_2$. Moreover, by Theorem \ref{t:existencerough}, the mean of $\omega$ is equal to $0$ and $\omega$ satisfies the energy identity~\eqref{e:energyid2}. Thus, by Lemma \ref{l:poinc},
		$$ \frac{1}{2}\frac{d}{dt}\|\omega(t)\|_{L^2}^2 =  - {\mathcal E}_\varphi(\omega(t),\omega(t))\leq   - \frac{1}{c_\varphi}\|\omega(t)\|_{L^2}^2.$$
		Then \eqref{e:l-alignment} follows by Gr\"onwall's inequality.
	\end{proof}
	
	\begin{prop}\label{p:stabadmis}
		Suppose that $c_\varphi<+\infty$ and let $(\rho_1, u_1)$ and $(\rho_2,u_2)$ be a pair of admissible solutions of~\eqref{e:EApuretopo} with equal means. Then we have
		\begin{equation}\label{e:stabadmis1}
			\int_{\R^2}|u_1(x,t)-u_2(y,t)|^2\dd\pi_t(x,y)\leq e^{-\frac{2t}{c_\varphi}}\int_{\R^2}|u_1(x,0)-u_2(y,0)|^2\dd\pi_0(x,y) 
		\end{equation}
		where $\pi_t(x,y) := (M^{-1}_1(\cdot,t),M^{-1}_2(\cdot,t))_\sharp\lambda$. In particular taking
		$$u_2=\bar{u}_1=\int_\R u_1(x,0)\dd\rho_1(x,0)$$
		and, noting that the $x$-marginal of $\pi_t$ is $\rho(\cdot,t)$, we obtain the exponential alignment estimate
		$$ \int_\R|u_1(x,t)-\bar{u}_1|^2\dd\rho_1(x,t)\lesssim e^{-\frac{2t}{c_\varphi}}\xrightarrow{t\to+\infty} 0.$$
	\end{prop}
	
	\begin{proof}
		By Proposition \ref{p:equivalence}, there exist admissible solutions $(v_1, M_1)$ and $(v_2, M_2)$ corresponding to $(\rho_1, u_1)$ and $(\rho_2, u_2)$, respectively, and which satisfy
		$$ u_1(M^{-1}_1(m,t),t)-u_2(M^{-1}_2(m,t),t) = v_1(m,t)-v_2(m,t).$$
		where $v_1-v_2$ is a weak solution to \eqref{e:reformulation} and satisfies the energy identity \eqref{e:energyid2}. By the 
		definition of $\pi_t$ and the pushforward change of variables formula we have, for every $t\ge0$,
		\begin{align*}
			\int_{\R^2} |u_1(x,t)-u_2(y,t)|^2\dd\pi_t(x,y)
			&= \int_0^1 
			\bigl|u_1(M^{-1}_1(m,t),t)-u_2(M^{-1}_2(m,t),t)\bigr|^2\dd m \\
			&= \int_0^1 |v_1(m,t)-v_2(m,t)|^2\dd m.
		\end{align*}
		The same identity holds at $t=0$. Therefore \eqref{e:stabadmis1} is
		exactly \eqref{e:l-alignment} rewritten in Eulerian variables.
		
		For the alignment estimate we choose $(\rho_2,u_2)$ to be the stationary
		state with constant velocity equal to the initial total momentum of
		$(\rho_1,u_1)$, namely
		\[
		\bar{u}_1 := \int_\R u_1(x,0)\dd \rho_1(x,0), \qquad
		u_2(x,t):= \bar{u}_1,\quad \rho_2 :=\rho_1.
		\]
		In mass coordinates,  this corresponds to
		\[
		v_2(m,t)\equiv \bar{u}_1,
		\qquad
		v_1(m,0) = u_1(M_1^{-1}(m,0),0),
		\]
		and by construction
		\[
		\int_0^1 v_1(m,0)\dd m
		= \int_0^1 u_1(M^{-1}_1(m,0),0)\dd m
		= \int_\R u_1(x,0)\dd\rho_1(x,0)
		= \bar{u}_1
		= \int_0^1 v_2(m,0)\dd m,
		\]
		which is then conserved for $t>0$, by  Theorem \ref{t:existencerough}.
		Plugging $u_2=\bar{u}_1$ into \eqref{e:stabadmis1} and observing that $(M^{-1}_1(\cdot,t))_\sharp\lambda=\rho_1(\cdot,t)$ yields the exponential alignment estimate.
	\end{proof}
	
	\begin{prop}\label{p:stabstrong}
		Let $(\rho_1, u_1)$ and $(\rho_2,u_2)$ be a pair of mass-distributional solutions of \eqref{e:EApuretopo} with equal means associated with $v$-strong solutions $(M_1,v_1)$ and $(M_2,v_2)$ of \eqref{e:reformulation}$_1$ under the protocol $\varphi(r) = C_sr^{-1-2s}$ with $s>\frac{1}{2}$, and the right side of \eqref{e:reformulation}$_1$ interpreted as $-\Lambda^s v$. Here $v_1,v_2$ are continuous representatives of their respective classes in $W^{s,2}(0,1)$. Then there exists a constant $\lambda_1>0$ such that for all $t\geq \tau>0$ we have
		$$ \sup_{x\in\R}|u_1(x,t)-u_2(x,t)|\lesssim e^{-\lambda_1 (t-\tau)}\|v_1(\tau)-v_2(\tau)\|_{W^{s,2}} $$
		In particular,
		$$\sup_{x\in\R}|u_1(x,t)-\bar{u}_1|\lesssim e^{-\lambda_1 t}\xrightarrow{t\to+\infty} 0, $$
		where
		$$ \bar{u}_1:=\int_\R u_1(x,0)\dd\rho_1(x,0). $$
		Moreover, if additionally $v_1(0),v_2(0)\in W^{s,2}(0,1)$ then we can take $\tau=0$ above.
	\end{prop}
	
	
	\begin{proof}
		The proof follows by a standard asymptotic estimate
		$$ \|v(t)\|_{C^{0,\alpha}}\xrightarrow{t\to+\infty} 0,
		$$
		for the fractional heat equation coupled with a direct change of variables performed {\it for all} $m\in(0,1)$. We sketch the proof for the reader's convenience, referring to \cite{BBC-2003} for details. 
		By linearity, $\omega:=v_1-v_2$ is a solution of \eqref{e:reformulation}$_1$ in the sense of Definition \ref{def:strongheat}.
		The inverse of the regional fractional Laplacian $S=(\Lambda^s)^{-1}$ in  $\{v\in L^2:\int_0^1v=0\}$ is a compact and self-adjoint operator on $L^2(0,1)$, which ensures the existence of an orthonormal basis $\{e_i\}_{i=1}^\infty$ of $L^2(0,1)$ consisting of eigenvectors of $S$ (with respective eigenvalues $1/\lambda_i$ associated with $e_i$). One can also show that the eigenvalue associated with constant functions is $\lambda_0=0$, while all others are positive, with the smallest determined by the Rayleigh quotient 
		$$0<\lambda_1
		:= \inf\left\{
		\frac{{\mathcal E}_\varphi(v,v)}{\|v\|_{L^2(0,1)}^2}
		\;:\;
		v\in W^{s,2}(0,1)\setminus\{0\},\ \int_0^1 v = 0
		\right\}.$$  In this basis,
		$$ \omega(t) = \sum_{i=1}^\infty e^{-\lambda_i t}\omega_i e_i,\quad t>0, $$
		and for $\tau\in(0,t)$ we have
		\begin{align*} 
			{\mathcal E}_\varphi(\omega(t),\omega(t)) & = (\Lambda^s \omega(t),\omega(t)) = \sum_{i=1}^{+\infty}\lambda_ie^{-2\lambda_i t}|\omega_i|^2 \\
			& \leq e^{-2\lambda_1 (t-\tau)}\sum_{i=1}^\infty\lambda_ie^{-2\lambda_i \tau}|\omega_i|^2 = e^{-2\lambda_1 (t-\tau)}{\mathcal E}_\varphi(\omega(\tau),\omega(\tau)).
		\end{align*} 
		Since $\omega(\tau)\in W^{s,2}(0,1)$ for a.e. $\tau\in(0,t)$, we conclude that
		\begin{equation}\label{e:maxstab.1}
			{\mathcal E}_\varphi(\omega(t),\omega(t))\leq e^{-2\lambda_1 (t-\tau)}{\mathcal E}_\varphi(\omega(\tau),\omega(\tau))\leq e^{-\lambda_1 (t-\tau)}\|\omega(\tau)\|_{W^{s,2}}^2.        
		\end{equation}
		Finally, a Campanato-Morrey-type inequality (cf. Theorem 8.2 from \cite{SobolevHitchhiker}) ensures that $\omega$, being continuous by definition, satisfies the H\"older estimate
		$$ \|\omega\|_{C^{0,\alpha}}^2\leq c_{m}{\mathcal E}_\varphi(\omega,\omega), \quad \alpha=s-\frac{1}{2}$$
		which together with \eqref{e:maxstab.1} yields in particular
		$$ \sup_{m\in(0,1)}|v_1(m,t)-v_2(m,t)|\lesssim e^{-\lambda_1 t}\|v_1(\tau)-v_2(\tau)\|_{W^{s,2}}. $$
		In particular, since, by Remark \ref{r:heatenergy}, the mean
		$$ \bar{v}_1 = \int_0^1v_1(m,t)\dd m =\int_0^1 v_1(m,0)\dd m = \int_\R u_1(x,0)\dd\rho_1(x,0) = \bar{u}_1$$
		is a constant-in-time solution to \eqref{e:reformulation}$_1$, we can take $v_2 = \bar{v}_1$ to arrive at
		$$ \sup_{m\in(0,1)}|v_1(m,t)-\bar{v}_1|\lesssim e^{-\lambda_1 t} \xrightarrow{t\to\infty} 0. $$
		Now since all values of $u(\cdot,t)$ are inherited from values of $v(\cdot, t)$ we obtain the desired stability and alignment estimates.
		
		\noindent
		Finally we note that for $\omega(0)\in W^{s,2}(0,1)$, throughout the proof we may take $\tau=0$ since then
		$$ {\mathcal E}_\varphi(\omega(0),\omega(0)) = \sum_{i=1}^\infty\lambda_i|\omega_i|^2 \leq \|\omega(0)\|^2_{W^{s,2}}<+\infty. $$

	\end{proof}
	
	\section{Conclusions and Outlook}
	
	The main goal of this paper was to investigate the relationship between purely topological alignment dynamics and its reformulation in mass coordinates, which leads to a decoupling of the velocity equation. We focused on well-posedness and long-time behavior of solutions corresponding to both bounded and singular communication kernels. In the latter case, the velocity equation reduces to a regional fractional heat equation.
	
	Part of the appeal of purely topological alignment dynamics lies in the simplicity of the arguments used to establish stability and alignment estimates, as demonstrated in Section~\ref{s:LongTime}. This naturally raises the question of long-time behavior in the case of short-range purely topological communication protocols, where the support of the interaction kernel $\varphi$ is contained in some interval $[0,\rho]\subset[0,1)$. In such a setting, the condition $c_\varphi<+\infty$, required for the Poincar\'e inequality in Lemma~\ref{l:poinc}, is no longer satisfied. 
	
	Nevertheless, alternative approaches based on chain connectivity arguments (see, for instance, \cite{MPT2019, ShvydkoySurvey} and related works of Tadmor \cite{T-21, T-23}) have been successfully employed to propagate interactions from short to long ranges through regions of positive density. Interestingly, in the purely topological setting, interactions appear to propagate even through regions of very low density, $\rho\ll\varepsilon$, suggesting a more accessible mechanism that merits further investigation.
	
	Another natural direction for future research concerns the mean-field limit. On the one hand, the velocity equation is fully decoupled from the rest of the system and, consequently, its evolution is insensitive to whether it transports an atomic measure or a continuum limit of empirical measures. Moreover, the scalar conservation law governing the cumulative distribution function $M$ admits approximations by piecewise constant functions $M^N$ (for instance, via a wave-front tracking argument), which are naturally associated with empirical measures $\rho^N=\partial_x M^N$. On the other hand, even with this structural information, it remains highly non-trivial to determine whether the corresponding particle system describing the evolution of the atoms of $\rho^N$ satisfies a closed system of ODEs of the form~\eqref{e:CS}.
	
	Finally, following Remark \ref{r:nowynew}, it is a valid question if local existence of classical (or at the very least, amissible) solutions with singular protocol $\varphi(r)=C_s s^{-1-2r}$ can be obtained. Specifically, one may aim for a threshold-type result akin to Theorem \ref{t:classical}.
	
	These considerations lead to the following questions:
	\begin{enumerate}
		\item[Q1] Is it possible to employ chain connectivity arguments to propagate interactions induced by short-range communication protocols without imposing assumptions on the magnitude of the agents' density?
		\item[Q2] Can mass-distributional solutions be rigorously constructed as mean-field limits of solutions to ODE systems of the form~\eqref{e:CS}?
		\item[Q3] Is there a variant of Theorem \ref{t:classical} for singular communication $\varphi(r)=C_s s^{-1-2r}$, providing sufficient condition for local existence of classical solutions?
	\end{enumerate}
	
	We believe that addressing these questions will contribute to a deeper understanding of purely topological alignment dynamics and help bridge the gap between particle-based models and their continuum descriptions.

	\section*{Acknowledgment}
	TL is supported in part by NSF grant DMS-2408585. JP's work is supported by the Polish National Science Centre’s Grant No. 2025/58/E/ST1/00482 (Sonata Bis).
	
	\section{Appendix}
	
	\begin{prop}
		\label{p:CDFprops}
		Let $\rho$ be a probability measure on $\mathbb{R}$. Let $M$ denote its right-continuous cumulative distribution function, and let $M^{-1}$ denote the left-continuous generalized inverse of $M$.  Let $\lambda$ denote Lebesgue measure on $(0,1]$.  Then the following statements hold.
		\begin{enumerate}[label = (\alph*)]
			\item $(M^{-1})_{\sharp} \lambda = \rho$.
			\item $M^{-1}\circ M(x) = x$ for $\rho$-a.e. $x\in \mathbb{R}$.  (In particular, $M^{-1}(M(x))\le x$ for all $x$, and strict inequality holds if and only if $\rho((w,x]) = 0$ for some $w<x$, that is, if $x$ has a vacuum interval directly to its left and is not an atom for $\rho$.) 
			\item $M_{\sharp} \rho = \lambda$ if and only if $M$ is continuous, if and only if $M\circ M^{-1}$ is the identity map on $(0,1)$.
			\item Suppose $f_1$ and $f_2$ are bounded, measurable functions.  If $f_1 \circ M^{-1} = f_2 \circ M^{-1}$ as elements of $L^1([0,1], \lambda)$, then $f_1 = f_2$ as elements of $L^1(\mathbb{R}, \rho)$.  If $M$ is continuous, then the converse implication holds as well.  
			\item If $f$ is bounded and measurable, then $\int_0^{M(x)} f\circ M^{-1} \dm = \int_{(-\infty, x]} f\rho \dx$ for all $x\in \mathbb{R}$.
		\end{enumerate}
	\end{prop}
	
	\begin{proof}
		(a) is standard and the proof is omitted.  (A reader interested in the details can find one in \cite{BolleyBrenierLoeper}.) Details of (b) and (c) are left to the reader.  
		We give the short proofs of (d) and (e).
		
		\medskip
		(d) Fix $\psi\in C_c(\mathbb{R})$.  The following short computation establishes the first claim.
		\begin{align*} 
			\int_{\mathbb{R}} \psi(x) f_1(x) \rho(x)\dx 
			& = \int_{[0,1]} \psi(M^{-1}(m)) f_1(M^{-1}(m)) \dm \\
			& = \int_{[0,1]} \psi(M^{-1}(m)) f_2(M^{-1}(m)) \dm 
			= \int_{\mathbb{R}} \psi(x) f_2(x) \rho(x)\dx. 
		\end{align*} 
		If $M$ is continuous, then $M\circ M^{-1}$ is the identity map on $(0,1)$, so that for any $\eta\in C_c((0,1))$ we have (whenever $f_1 = f_2$ as elements of $L^1(\rho)$) that
		\[
		\int_{[0,1]} \eta(m) f_1 \circ M^{-1}(m)\dm 
		= \int_{\mathbb{R}} (\eta\circ M)f_1 \rho \dx 
		= \int_{\mathbb{R}} (\eta\circ M)f_2 \rho \dx 
		= \int_{[0,1]} \eta(m) f_2 \circ M^{-1}(m)\dm 
		\]
		We conclude that $f_1\circ M^{-1} = f_2\circ M^{-1}$ as elements of $L^1([0,1])$ in this case.
		
		(e) We note that
		\[
		(0,M(x)) \subset \{m\in (0,1]: M^{-1}(m)\le x\} \subset (0, M(x)].
		\]
		Therefore 
		\[
		\int_0^{M(x)} f\circ M^{-1}(m)\dm = \int_0^1 \big(f1_{(-\infty, x]})\circ M^{-1}(m) \dm = \int_\mathbb{R} (f1_{(-\infty,x]})(y) \rho(y)\dy = \int_{(-\infty, x]} f\rho \dy.
		\]
	\end{proof}


	\def\cprime{$'$}

\end{document}